\documentclass[sn-mathphys-num]{sn-jnl}

\usepackage{graphicx}%
\usepackage{multirow}%
\usepackage{amsmath,amssymb,amsfonts}%
\usepackage{amsthm}%
\usepackage{mathrsfs}%
\usepackage[title]{appendix}%
\usepackage{xcolor}%
\usepackage{textcomp}%
\usepackage{manyfoot}%
\usepackage{booktabs}%
\usepackage{algorithm}%
\usepackage{algorithmicx}%
\usepackage{algpseudocode}%
\usepackage{listings}%

\usepackage{enumerate,esint,mathtools,xparse}

\usepackage{bbm}

 \usepackage{mathrsfs} 
\usepackage{mathabx}

\theoremstyle{thmstyleone}%
\newtheorem{theorem}{Theorem}

\newtheorem{corollary}[theorem]{Corollary}
\newtheorem{lemma}[theorem]{Lemma}
\newtheorem{proposition}[theorem]{Proposition}

\newcommand\ve{\varepsilon}

\newcommand{\R}{\mathbb{R}}
\newcommand{\E}{\mathbb{E}}

\newcommand{\T}{\mathbb{T}}

\renewcommand{\div}{\mathop{\rm div}\nolimits}

\newcommand{\td}{\tilde}

\renewcommand{\bar}[1]{\overline{#1}}

\newcommand{\hd}{\widehat{\textbf{d}}}

\newcommand{\pa}{\partial}
\newcommand{\Q}{\mathcal{Q}}

\renewcommand*{\tilde}{\widetilde}
\everymath{\displaystyle}

\newcommand{\uone}{{\bf{z}}}
\newcommand{\utwo}{{\bf{v}}}
\newcommand{\pone}{Q}
\newcommand{\ptwo}{\pi}
\newcommand{\rd}{\,\mathrm{d}}
\newcommand{\uu}{{\bf{u}}}
\newcommand{\dd}{{\bf{d}}}
\newcommand{\hP}{\widehat {\ptwo}}

\NewDocumentCommand{\oldnorm}{som}{%
  \IfBooleanTF{#1}
    {\oldnormaux*{#3}}
    {\IfNoValueTF{#2}
       {\oldnormaux*{\vphantom{dq}#3}}
       {\oldnormaux[#2]{#3}}%
    }%
}

\usepackage{scalerel}

\allowdisplaybreaks

\theoremstyle{thmstyletwo}%
\newtheorem{remark}{Remark}%

\theoremstyle{thmstylethree}%
\newtheorem{definition}{Definition}%

\begin{document}

\title[Article Title]{Partial Regularity for the Three-Dimensional Stochastic Ericksen--Leslie Equations}


\author*[1]{\fnm{Hengrong} \sur{Du}}\email{hdu@fisk.edu}

\author[2]{\fnm{Chuntian} \sur{Wang}}\email{cwang27@ua.edu}


\affil*[1]{\orgdiv{Department of Mathematics and Computer Science}, \orgname{Fisk University}, \orgaddress{\street{205 Dubois Hall}, \city{Nashville}, \postcode{37208}, \state{TN}, \country{USA},} ORCID: 0000-0003-2392-8963}

\affil[2]{\orgdiv{Department of Mathematics}, \orgname{The University of Alabama}, \orgaddress{\city{Tuscaloosa}, \postcode{35487}, \state{AL}, \country{USA},} ORCID: 0000-0002-3888-7589}


\abstract{In this article, we investigate the global existence of martingale suitable weak solutions to stochastic Ericksen--Leslie equations with additive noise
 in a 3D torus.   The notion of suitable weak solutions
 {has} been introduced to     address    possible emergence of finite-time singularities, which   remains  a 
 {notably}
challenging question in the field of fluid dynamics. 
Weak solutions offer an approach to account for these potential singularities. A restricted class of weak solutions that exhibit a higher level of regularity which are therefore more likely to be physically meaningful, is naturally called for. 
Consequently,    
  \textit{suitable weak solutions},  
i.e., 
weak solutions that satisfy a local energy inequality, become a focus of research, including   {investigations} about  how regular   these  solutions can be.   
In this article, we prove that,  despite 
 {the presence of}
 white noise, the paths of martingale suitable weak solutions of 3D stochastic Ericksen--Leslie equations exhibit singular points of one-dimensional parabolic Hausdorff measure zero.
To  establish this result,  we have  {utilized} two techniques,  which can potentially be generalized to handle other   {stochastically forced} complex fluid dynamics equations with a similar
  structure.
 Firstly,  a local energy-preserving approximation is constructed which markedly facilitates   the proof of the  global     existence of martingale suitable weak solutions; secondly,  to demonstrate   partial regularity of these solutions,   a blow-up argument is formulated, 
which efficiently yields the desired key estimate.
}

\keywords{Stochastic PDEs, Partial regularity, Complex fluid  {dynamics}, Hydrodynamics of liquid crystal}


 \pacs[MSC Classification]{60H15, 35R60, 60H30, 76D06, 76D03}

\maketitle

\section{Introduction}
\label{sec:intro}

Liquid crystal, an anisotropic phase situated between the isotropic liquid phase and solid phase, has 
drawn considerable attention from researchers
  from various disciplines. Among all types of liquid crystal states, the nematic phase is particularly interesting for its characteristic that rod-shaped molecules display the orientational order (tending to point in the same direction) without positional order.  
The  Ericksen--Leslie equations  (see for instance \cite{Ericksen1961, Ericksen1966,Leslie1968}) serve as a fundamental mathematical model for describing the hydrodynamics of nematic liquid crystal flows, 
  a subject of extensive research by itself  \cite{
lin2014recent, xu2022recent,  zarnescu2021mathematical}.  To characterize  orientations of nematic liquid crystal molecules, `order parameters' are commonly adopted;  the so-called `director theory', which 
 {covers}
%
 {the Ericksen--Leslie theory \cite{lin1989nonlinear} as a special case,   employs unit vector field   to represent alignment of the molecules. }

%

The Ericksen--Leslie equation system   typically   combines momentum balance with the evolution of molecular orientation, captured by the Navier--Stokes equations  \cite{Foias2001, Temam4,  Temam1} and  the heat flow of harmonic maps \cite{Linwang2008},  respectively.  
In this article  we focus on the \emph{simplified} version of the full Ericksen--Leslie equation system,  
 which was initially studied in a series of works   including 
\cite{Lin1995, lin1995partial}:
\begin{equation}
  \left\{
    \begin{array}{l}
      \rd {\uu} ^ \epsilon-\Delta {\uu} ^ \epsilon \rd t+{\uu} ^ \epsilon \cdot \nabla {\uu} ^ \epsilon\rd t+\nabla P ^ \epsilon \rd t =-\div(\nabla {\dd} ^ \epsilon\odot \nabla {\dd} ^ \epsilon)\rd t , \\
      \div {\uu} ^ \epsilon=0, \\
      \rd {\dd} ^ \epsilon-\Delta {\dd} ^ \epsilon\rd t+{\uu} ^ \epsilon\cdot \nabla {\dd} ^ \epsilon\rd t=-f_\epsilon({\dd} ^ \epsilon)\rd t.
    \end{array}
    \right.
    \label{eqn:SEL0p}
\end{equation}
Here with every fixed $\epsilon>0$, the unknown triple 
 $({\uu} ^ \epsilon, P ^ \epsilon, {\dd} ^ \epsilon)$ represents the fluid velocity field, the pressure function, and the director field which characterizes the mean orientation of molecules,  respectively.  Moreover  
\begin{equation}\label{eq:ginz2}
f_\epsilon({\dd}^ \epsilon)=\nabla_{\dd} F_\epsilon({\dd}^ \epsilon), 
\end{equation}
  where
\begin{equation}\label{eq:ginz1}
  F_\epsilon({\dd}^ \epsilon)=\frac{1}{\epsilon^2}(|{\dd}^ \epsilon|^2-1)^2.
\end{equation}
It is evident that  \eqref{eqn:SEL0p}$_{1, 2}$ is a forced  Stokes system with the so-called Ericksen stress tensor 
\begin{equation*}
(\nabla {\dd}^ \epsilon\odot \nabla {\dd}^ \epsilon)_{ij}=\pa_{x_i} {\dd^ \epsilon}\cdot \pa_{x_j}{\dd}^ \epsilon,  
\end{equation*}   while \eqref{eqn:SEL0p}$_3$ is a transported semi-linear heat flow associated to the Ginzburg--Landau potential function  $F_\epsilon({\dd}^ \epsilon)$.

 Subsequently   we will only consider (\ref{eqn:SEL0p}) with  a fixed $\epsilon$; without loss of generality  we will fix $\epsilon=1$ and omit the corresponding subscripts and superscripts. 
 Indeed exploration of the singular limit as $\epsilon\to 0$ in \eqref{eqn:SEL0p}$_{3}$ is an ultimate goal,  which  leads  to the  transported heat flow of harmonic maps for the director field $\dd$ as follows:
\begin{equation*}
 \rd {\dd}-\Delta {\dd}\rd t+{\uu}\cdot \nabla {\dd}\rd t=|\nabla {\dd}|^2 {\dd} \rd t.
 \end{equation*}

  It is well-established that the critical spatial dimension 
 for heat flow of
 harmonic maps 
 is two 
 \cite{Chang1992, Chen1988}.   Consequently,  in space dimension two, it has been
 confirmed that solutions to the  (limited) Ericksen--Leslie
 equation are partially regular \cite{Lin2010}, and singularities
 could occur \cite{Huang2016, Lai2022}.   Moreover, in space dimension two,     the singular limit as
 $\epsilon \to 0$ has been justified through concentration-cancellation
 compactness \cite{Du2022,Kortum2020}.  However,
 in space dimension three, the global well-posedness for  the   (limited) Ericksen--Leslie equation
 {remains a} prominent challenge,   with only a handful of partial results available, 
such as those presented in \cite{hieber2016, hieber2017, hieber2018modeling,hieber2019, hieber Nesensohn2016,lin2016global, Wang2011}.
%
%
%

 While  the study of deterministic Ericksen--Leslie equation system has attracted  considerable attention,  
it is worth noting that the deterministic framework
could 
possibly
overlook  complexities introduced by real-world scenarios, including environmental disturbances \cite{Chorin2013}, measurement uncertainties, and inherent thermal fluctuations \cite{Bhattacharjee2010, Lee2015}. To address these intricacies,  a typical way is to introduce a stochastic component into the system. 
%
Stochastic forcing in  the nematic liquid crystal flows may potentially lead to    emergence of singularities, which  can 
 {cause}
 turbulence in fluid flow   \cite{Buckmaster2019, Guillod2017,  Hou2022,Tao2016} and  defects in molecular alignments \cite{DeGennes1993,   Lavrentovich2001,Virga2018}. 
These phenomena
 bear 
substantial implications,  impacting our understanding of both theory and the practical applications of liquid crystal materials.

To this end,   
in this article, we study the    stochastically forced  Ericksen--Leslie equations (SEL)  
  on the three-dimensional tori $\mathbb{T}^3$ driven by an
additive white noise $W$ as follows: 
\begin{equation}
  \left\{
    \begin{array}{l}
      \rd {\uu}-\Delta {\uu} \rd t+{\uu} \cdot \nabla {\uu}\rd t+\nabla P \rd t =-\div(\nabla {\dd}\odot \nabla {\dd})\rd t+\rd W, \\
      \div {\uu}=0, \\
      \rd {\dd}-\Delta {\dd}\rd t+{\uu}\cdot \nabla {\dd}\rd t=-f ({\dd})\rd t.
    \end{array}
    \right.
    \label{eqn:SEL0}
\end{equation} 
 {Here we note again that  $f ({\dd}) $ ccorresponds to $ f_\epsilon({\dd}^ \epsilon)$ when $\epsilon$ is fixed as 1 (see also Equations \eqref{eq:ginz2} and \eqref{eq:ginz1}). 
Our main objectives here are two-fold. 
Firstly,  we introduce the concept of     martingale \emph{suitable} weak solutions to \eqref{eqn:SEL0} and establish their  global   existence.   
%
Secondly, we demonstrate that  under certain relatively reasonable assumptions on $W$ (see Section \ref{sec:sto}), the set of singular points, where solutions  {become} unbounded,  cannot have a positive one-dimensional parabolic  {Hausdorff} measure. We remark  that the concept of martingale solutions that we formulate here is equivalent to that of statistical solutions,  a topic  extensively studied  in fluid dynamics (see for instance \cite{foias2010note,foias2013properties,foias2019properties}). 

%
%
%
 {
It is worthwhile to discuss the physical relevance of our choice of modeling the stochastic nematic liquid crystals dynamics  by introducing noise into the fluid velocity, as   in (\ref{eqn:SEL0}).  For example, 
  dynamic scattering,
 a well-documented and extensively studied physical phenomenon observed in nematics (see, e.g., \cite{g19687, g1968, Y2018, A1997, Yo1977, hui2015}),  
typically arises in experiments
 where  the  cellular flow of the   liquid     intensifies to the point of becoming turbulent. Here the liquid velocity plays a central role in the system dynamics, heavily affecting the alignment  of the optical director  (see e.g. \cite{DeGennes1993}).
As a result, it is reasonable to treat liquid flow as the main driver of the system dynamics and assume   velocity fluctuations as the dominant source of randomness.
%
%
%
Moreover, 
there are
other electrohydrodynamic instabilities   identified in previous studies  (see e.g. \cite{B1985, Ve1991, Ve1988, Ve1989, Di1975, wim1971}) that  also support the class of SEL models   in which   liquid flow plays a central role.
%
%
%
%
A classical example  is the Williams domain instability. Visually resembling  the Rayleigh-Bénard instability in isotropic fluids, 
 it
can produce a variety of  flow geometries depending on the angle between the wave vector and the director field,  including normal rolls, oblique rolls, and more complex structures \cite{H2022, Hirata2022, r1973}.
Additionally, some recent     works  on stochastic nematic liquid crystal systems also employ this SEL modeling framework where noise is introduced in the velocity component; see, for example, \cite{guo2019, qiu2022}. 
%
 }

%

 The question of whether fluid dynamics equations such as  3D Navier-Stokes equations develop finite-time singularities has been a well-known open problem for  a long time;   to take into account   these singularities,   it becomes necessary to consider weak solutions,   which  live in a larger function space than  that of continuous or differentiable functions. 
However,  issues may arise with
weak solutions,  such as    non-uniqueness and non-physical properties (see e.g. \cite{albritton2022non,   Buckmaster2019}),  due to    insufficient regularity. 
%
 %
 %
  With this motivation,  the concept of  suitable weak solutions   which  satisfy a local energy inequality is introduced  (for the first time in \cite{leray1934mouvement}).
Subsequently, a natural question arises: how regular can these solutions be?  
%
%
Various results have been  found for the Navier--Stokes equation,  as in \cite{scheffer1976partial,  scheffer1977hausdorff},  where suitable weak solutions are  shown to   be smooth except for a singular set of parabolic 
 {Hausdorff}
 dimension $5/3$.  Later on in   \cite{Caffarelli},  this result was   improved;  the dimension of this singular set is demonstrated to be no greater than 1. 
Since then,    numerous  works have been carried out to explore   partial regularity in the context of  general fluid dynamics equations,  see e.g.  \cite{constantin2001some, guo2002partial,  hyder2022partial}.  Notably,   in \cite{lin1995partial},     a partial regularity  result for the  deterministic Ericksen--Leslie 
 {equations} 
(\ref{eqn:SEL0}) in space-dimension three was achieved. 

In contrast,  there are relatively  limited findings regarding the  partial regularity for stochastic PDEs.
%
%
%
  To the best of our knowledge,   such investigations have primarily focused on the stochastic Navier-Stokes equations; notably,   in \cite{Flandoli2002,   Romito2010},   it was established that in 3D the same partial regulairy result   as in \cite{Caffarelli} holds for martingale suitable weak solutions for almost surely every random path. Building upon the  insights of \cite{Flandoli2002, Romito2010},  here in this article we establish the  same type of partial regularity result  for  SEL in    space dimension three.
As far as we know, there has been few  reported results  about the regularity of 
the SEL equations in space dimension three,  even with $\epsilon $ being fixed (\cite{Brzezniak2019, Brzezniak2020a}),  despite a  growing interest in this topic recently. In  space dimension two,  the well-posedness results concerning the  SEL equations
have been founded in   \cite{DeBouard2021} (with additive noise)   and   \cite{wanglidan2023} (with multiplicative noise), 
and for further details  see e.g. \cite{Brzezniak2020,   Medjo2019}.

The article is organized as follows. 
In Section \ref{sec:math},  we specify notations and stochastic framework needed to set up the SEL system.  Then in Section \ref{sec:stoc} we 
introduce the notion of suitable weak solutions, 
drawing inspirations 
from a split-up scheme  which is introduced in
\cite{Flandoli2002,  Romito2010}   
in the context of the
stochastic Navier--Stokes equation.
Specifically,   this approach involves splitting
solutions to the SEL equations
into two components.  The first part is  
governed by 
linear stochastic Stokes equations,  and the second part is subject to the remaining nonlinear coupled system. For   the first component of the solution, we  directly obtain a desired bound, whenever    a 
%
 the noise term is bounded in a suitable fractional Sobolev space almost surely  (Section \ref{sec:def}).
 Subsequently over each random path that satisfies this assumption,  we define   deterministic suitable weak solutions which 
satisfy 
two local energy inequalities  (Section \ref{sec:def}).    Thus  we come up with the notion of  a martingale suitable weak solution  as a martingale solution  each   trajectory of which  is almost surely   a deterministic suitable weak solution. Next in Section \ref{sec:main},  the main result  of the article is  stated, namely, the global    existence and  partial regularity of martingale suitable weak solutions. These two claims are proved in Sections \ref{sec:exis} and \ref{sec:partial},  respectively. 
%
For the proof of the global
existence
 (Section \ref{sec:exis}), 
it is worth noting that the main challenge is to construct  an  approximation  such that the local energy inequality  is preserved after passage to the limit.  Here we 
employ a 
specific
 type of approximation by regularization, which was introduced in \cite{leray1934mouvement} for the Navier--Stokes equations. In \cite{Caffarelli},  a similar approximation is adopted through a  'retarded mollification' to establish the existence of suitable weak solutions for the Navier--Stokes equations.

\section{Mathematical settings and notations}
\label{sec:math}

\subsection{Functional spaces}
\label{sec:not}
We introduce two fundamental   functional spaces that take into account the incompressibility condition \eqref{eqn:SEL0}$_2$ (cf. \cite{Temam1}):
$$H:=\{\uu\in L^2(\T^3; \R^3)|\div \uu=0\},$$
and 
$$V:=\{\uu\in H^1(\T^3; \R^3)|\div \uu=0\}.$$
The operator $A:D(A)\subset H\to H$ is defined as $Au=-P_L\Delta \uu$, where $P_L$ is the Helmholtz--Leray projection from $L^2(\T^3; \R^3)$ onto $H$ and $D(A)=H^2(\T^3; \R^3)\cap V$. It is well-known that $A^{-1}$ is well-defined from $H$ to $D(A)$. It is also self-adjoint. By the spectral theory of compact self-adjoint operators in a Hilbert space,  there exists a sequence of eigenvalues  $0<\lambda_1\le \lambda_2\le \cdots\le \lambda_m\le \cdots$, $\lambda_m\to \infty$ as $m\to \infty$ and an orthonormal basis $\{e_i\}_{i=1}^\infty$ such that $Ae_i=\lambda_i e_i$.  
The fractional power $A^\alpha$ of $A$, $\alpha\ge0$, 
is
 simply defined by 
$$A^\alpha \uu=\sum_{i=1}^{\infty}\lambda_i^\alpha e_i ( \uu, e_i)_{L^2(\T^3)} , $$
with 
the 
domain 
$$D(A^\alpha)=\left\{\uu\in H\Big|\|\uu\|_{D(A^\alpha)}^2=\sum_{i=1}^{\infty}\lambda^{2\alpha}_i( \uu, e_i)_{L^2(\T^3)}^2 =\|A^\alpha \uu\|_{L^2(\T^3)}^2<\infty\right\}.$$
We also
 note that $D(A^\alpha)\subset H^{2\alpha}(\T^3; \R^3)$. For the definition of $H^{2\alpha}(\T^3;\R^3)$, we refer to \cite[Definition 1.10]{Robinson2016NS}.

\subsection{Stochastic framework}\label{sec:sto}
In order to define the noise term,  we first recall some basic concepts and notations of probability and stochastic processes.  For more details see for instance 
\cite{ZabczykDaPrato1,Flandoli1, FlandoliGatarek1, Flandoli2002,    pezat2007, Romito2010}.

We fix a probability space $(\Omega, \mathcal F, \mathbb P)$,  where $\Omega$ is a sample space equipped with a $\sigma$-algebra $\mathcal F$, and a probability measure $\mathbb P$. Then we define a
stochastic basis \begin{equation}\label{stochasticbasis}
\mathfrak{B}:= \left( \Omega, \mathcal F, \{\mathcal {F}_t \}_{t\geq 0}, \mathbb P,   \left (W(t) \right)_{t\geq 0} \right),
\end{equation}
which is a filtered probability space with a Brownian motion $ \left (W(t) \right)_{t\geq 0}  $    adapted to the filtration  $\{\mathcal {F}_t \}_{t\geq 0}$. In order to avoid unnecessary complications below we may assume that $\{\mathcal {F}_t \}_{t\geq 0}$ is complete and right continuous (see \cite{pezat2007}). 
We then assume that 
\begin{equation}\label{eq:assum4}
\mbox{$\left (W(t) \right)_{t\geq 0}$
  takes value in
 $D(A^\delta)$ for some $\delta>0$ with a covariance operator $\mathcal{L}$.}
\end{equation}
Here $\mathcal{L}\in \mathcal{B}(D(A^\delta))$ and 
is non-negative and of trace class. Let $\{b_i\}_{i\geq 1}$ be the orthonormal basis of $D(A^\delta)$ consisting of eigenvectors of $\mathcal{L}$ subject to non-negative eigenvalues $\{\gamma_i\}_{i\geq 1}$. Then we have the following expansion (cf. \cite[Theorem 4.20]{pezat2007})
$$W(t)=\sum_i W_i(t)b_i,
$$
where the real-valued
Brownian motions
$$W_i(t)=(W(t), b_i)_{D(A^\alpha)}$$
are independent and their covariances are
$$\mathbb{E}[W_i(t)W_i(s)]=(t\wedge s)\gamma_i.$$
Moreover, by Kolmogorov's continuity criterion (see for instance \cite{stroock1979}), we have that for every $\beta  \in (0,  1/2)$, it holds that 
\begin{equation}\label{eqn:assum2}
W(t) \in  C^{\beta} \left (  [0, T] ;  D (A^\delta) \right) \quad \mathbb P- a.s..
\end{equation}
 {For the purpose of obtaining a key preliminary estimation later on (see Lemma \ref{lemma:boundz} in Section \ref{sec:def}), we will need to further assume that 
$\delta\geq\frac{3}{4}$. 
So
this assumption
 together with (\ref{eq:assum4}) 
yields
the assumption that we require on the noise: 
\begin{equation}\label{eq:assum5}
\mbox{$\left (W(t) \right)_{t\geq 0}$
  takes value in
 $D(A^\delta)$ for some $\delta>3/4$ with a covariance operator $\mathcal{L}$.}
\end{equation}}

\section{Deterministic and martingale suitable weak solutions}
\label{sec:stoc}

In order to define  suitable weak solutions that are deterministic and are martingale,     we follow the approach of \cite{Flandoli2002,  Romito2010},  and  split the SEL   into the sum of a linear stochastic Stokes equation  and a  nonlinear modified Ericksen-Leslie equation.

\subsection{A split-up scheme of the SEL}

We will assume throughout the paper the following assumptions 
on
the initial data
\begin{equation}
  \begin{aligned}
  &({\uu}_0, {\dd}_0)\in H\times H^1(\T^3;\R^3),\quad  |{\dd}_0(x)|=1 \text{ for all }x\in \T^3. 
\label{eqn:assum}
  \end{aligned}
\end{equation} 
As a preparation to introduce notions of deterministic and stochastic suitable
weak martingale  solutions,  we will first split the solutions ${\uu}$ and $P$
properly. 
%
We introduce new variables ${\uu}={\uone}+{\utwo}$, $P={\pone}+{\ptwo}$,  where $({\uone}, {\pone})$ solves the linear stochastic Stokes system as follows
\begin{equation}
  \left\{
    \begin{array}{l} {\rd} {\uone}-\Delta {\uone} {\rd}t +\nabla {\pone} {\rd}t ={\rd} W  , \\
      \div {\uone}=0, \\
      {\uone}(0)=0.
    \end{array}
    \right.
    \label{eqn:SELu1}
\end{equation}
The triple $({\utwo}, {\ptwo}, {\dd})$ solves the modified Ericksen--Leslie equations as follows
\begin{equation}
\begin{cases}
   &   \pa_t {\utwo}-\Delta {\utwo}+({\uone}+{\utwo})\cdot\nabla ({\uone}+{\utwo})+\nabla {\ptwo}\\
&\quad \quad \quad \quad \quad  = -\div\left( \nabla {\dd}\odot \nabla {\dd}-\frac{1}{2}|\nabla {\dd}|^2 {\rm I}_3-F({\dd}){\rm I}_3 \right), \\
   &   \div {\utwo}=0, \\
     & \pa_t {\dd}-\Delta {\dd}+({\uone}+{\utwo})\cdot \nabla {\dd}=-f({\dd}), \\
   &   {\utwo}(0)={\uu}_0, {\dd}(0)={\dd}_0. 
\end{cases}
    \label{eqn:SELu2}
\end{equation}
\begin{remark}\label{rem1}
	Thanks to the incompressibility condition \eqref{eqn:SELu2}$_2$,  the 
 right-hand side of \eqref{eqn:SELu2}$_1$ is is equivalent to the deterministic term on the right-hand side of
\eqref{eqn:SEL0}$_1$. 
 Specifically, for the extra term, we observe:
\begin{equation*}
\displaystyle\div\left(\frac{1}{2}|\nabla {\dd}|^2{\rm I}_3 +F({\dd}){\rm I}_3\right)=\nabla\left(\frac{1}{2}|\nabla {\dd}|^2+F({\dd})\right),
\end{equation*}
 which can be treated as part of the pressure gradient.
Moreover, this observation,  together with direct computations,  yields
 \begin{equation}
     \label{eqn:RHSu}
     -\div\left( \nabla {\dd}\odot \nabla {\dd}-\frac{1}{2}|\nabla {\dd}|^2 {\rm I}_3-F({\dd}){\rm I}_3 \right)=-\nabla \dd\cdot (\Delta \dd-f(\dd)).
 \end{equation}
 {As a result, 
%
for computational convenience, we shall use the right-hand-side of the above equation to replace the    right-hand side of of \eqref{eqn:SELu2}$_1$.
}
\end{remark}


\subsection{Definitions of suitable weak solutions}
\label{sec:def}


 In order to  define what it means for the modified Ericksen-Leslie equation  (\ref{eqn:SELu2}) to have a weak solution,  we need  a preliminary estimation of the solution to the linear stochastic Stokes system \eqref{eqn:SELu1}. We first derive a   result that ensures the boundedness of ${\uone} (t)$, 
%
the stochastic process that solves \eqref{eqn:SELu1}. 
To achieve  this goal, we   need to 
 {utilize}
 the following result, which is also in \cite{Flandoli2002} (see  \cite[Lemma 2.2]{Flandoli2002}). 
\begin{lemma}[{An  almost-sure bound of the linear stochastic Stokes system}]
\label{lemma:boundz}
 {With the assumption on the noise as in \eqref{eq:assum5}, we have }
%
%
\begin{equation}
\label{eqn:zregu}	
	{\uone}(\omega)\in  L_{\rm loc}^\infty(\mathbb{T}^3\times[0,\infty)) \mbox{ for $  \mathbb P$- a.s. $\omega$}.
\end{equation} 
\end{lemma}
The proof of Lemma \ref{lemma:boundz} is provided in Appendix \ref{secA1}. 
 %
%
%
}

With Lemma \ref{lemma:boundz} in hand, 
now we are ready to introduce the  notions of deterministic and martingale suitable weak solutions to SEL. 
\begin{definition}\label{def:weaksolution}
 Assume that $T>0$ and fix $\omega\in \Omega$ such that $\uone(\omega)$ satisfies
 {the bound specified in} \eqref{eqn:zregu}. Then 
we say $({\utwo}, {\dd})$ is a weak solution to (\ref{eqn:SELu2}) in the sense of distribution on 
\begin{equation*}
Q_T:=\T^3\times(0, T],
\end{equation*}
 subject to $ {\utwo}(0)={\uu}_0, {\dd}(0)={\dd}_0$, 
%
{if}   for any 
%
%
{$\phi_1 \in  C^\infty (\T^3, \R^3)$}, 
$\div\phi_1 =0$, 
%
%
{$\phi_2 \in C^\infty (\T^3, \R^3)$},
and $t>0$,  the following weak formulation of (\ref{eqn:SELu2})
is satisfied:
\begin{equation*}
    \begin{aligned}  
   \int_{\T^3 \times{\{t\}}}     {\utwo} \cdot\phi_1 - &\int_0 ^t\int _{\T^3}  {\utwo}\cdot\Delta \phi_1+({\uone}+{\utwo})\otimes({\uone}+{\utwo}) : \nabla \phi_1 \\
  & =
   \int _{\T^3} \uu_0 \cdot\phi_1  
    + \int_0 ^t\int _{\T^3}     \left( \nabla {\dd}\odot \nabla {\dd}-\frac{1}{2}|\nabla {\dd}|^2 {\rm I}_3-F({\dd}){\rm I}_3 \right):\nabla \phi_1, \\
   \int_{\T^3 \times{\{t\}}}     {\dd} \cdot\phi_2 -  &\int_0 ^t\int _{\T^3}  {\dd}\cdot\Delta \phi_2+({\uone}+{\utwo})\otimes \dd  : \nabla \phi_2  
    =
    \int _{\T^3} \dd_0\cdot \phi_2 
    - \int_0 ^t\int _{\T^3}  f(\dd)   \cdot  \phi_2.
    \end{aligned}
\end{equation*}
\end{definition}


\begin{definition}\label{def:suitweak}
Assume that $T>0$ and fix $\omega\in \Omega$ such that $\uone(\omega)$ satisfies
 {the bound specified in} 
 \eqref{eqn:zregu}. Then a suitable weak solution to \eqref{eqn:SEL0} is a triple $({\uu}, P, {\dd})$ such that if ${\utwo}={\uu}-{\uone}$ and ${\ptwo}=P-{\pone}$ where $({\uone}, {\pone})$ is the solution to \eqref{eqn:SELu1}, then 
  \begin{enumerate}
    \item $({\utwo}, {\dd})$ is weakly continuous with respect to time, 
    \item ${\utwo}\in L^\infty(0, T; H)\cap L^2(0, T; V)$, ${\dd}\in L^\infty(0, T; H^1(\T^3; \R^3))\cap L^2(0, T; H^1(\T^3; \R^3))$ and ${\ptwo}\in L^{5/3}(0,T; L^{5/3}(\T^3;\R)),$
   \item $({\utwo}, {\dd})$ is a weak solution to (\ref{eqn:SELu2}) in the sense of distribution on $Q_T$   as specified in Def. \ref{def:weaksolution},
    \item for all $t\le T$ and almost every $s<t$, we have
\begin{equation} \label{eq:ine1}
      \begin{aligned}
       & \int_{\T^3\times\{t\}}\left(\frac{1}{2}|{\utwo}|^2+\frac{1}{2}|\nabla {\dd}|^2+F({\dd})\right)+\int_s^t\int_{\T^3} |\nabla {\utwo}|^2+|\Delta {\dd}-f({\dd})|^2 \\
       \le  &\int_{\T^3\times\{s\}}\left( \frac{1}{2}|{\utwo}|^2+\frac{1}{2}|\nabla {\dd}|^2+F({\dd}) \right)\\
        &+\int_{s}^{t}\int_{\T^3}{\uone}\cdot [({\uone}+{\utwo})\cdot \nabla {\utwo}]+({\uone}\cdot \nabla){\dd}\cdot (\Delta {\dd}-f({\dd})).
      \end{aligned}  
\end{equation}
 \item for any $\varphi\in C_0^\infty(Q_T)$ ({we note   that the subscript “0”    refers to compact support in the corresponding domain}),  $\varphi\ge 0$, we have
\begin{equation}\label{eq:ine2}
      \begin{aligned}
        &\int_{\T^3\times\{T\}}\left( \frac{|{\utwo}|^2}{2}+\frac{|\nabla {\dd}|^2}{2}+F({\dd}) \right)+\int_{Q_T}(|\nabla {\utwo}|^2+|\Delta {\dd}|^2+|f({\dd})|^2) \varphi\\
       \le & \int_{Q_T}\left( \frac{|{\utwo}|^2}{2}+\frac{|\nabla {\dd}|^2}{2}+F({\dd}) \right)\pa_t \varphi+\int_{Q_T}[({\uone}+{\utwo})\cdot \nabla \utwo]\cdot \uone\varphi\\
        &+\int_{Q_T}\left(\frac{|{\utwo}|^2}{2}+\utwo\cdot \uone\right)[({\uone}+{\utwo})\cdot \nabla\varphi]+{\ptwo}{\utwo}\cdot \nabla\varphi+\left( \frac{|{\utwo}|^2}{2}+\frac{|\nabla {\dd}|^2}{2} \right)\Delta \varphi\\
        &+\int_{Q_T}\left( \nabla {\dd}\odot \nabla {\dd}-\frac{1}{2}|\nabla \dd|^2 {\rm I}_3 \right):\nabla^2 \varphi\\
        &+\int_{Q_T}({\uone}\cdot \nabla {\dd})\cdot (\Delta {\dd}-f({\dd}))\varphi+[({\uone}+{\utwo})\cdot \nabla {\dd}]\cdot (\nabla \varphi\cdot \nabla){\dd}\\
        &-\int_{Q_T}f({\dd})\cdot (\nabla \varphi\cdot \nabla){\dd}-2\int_{Q_T}\nabla f({\dd}):\nabla {\dd}\varphi.
      \end{aligned}
\end{equation}
  \end{enumerate}
\end{definition}
%
\begin{definition}\label{def:mar}
A martingale suitable weak solution to \eqref{eqn:SEL0} is a quadruple  $(\tilde {\mathfrak{B}} , \tilde {\uu},  \tilde P,  \tilde {\dd} )$,  
where $ \tilde {\mathfrak{B}}$ is a stochastic basis
\begin{equation}\label{eq:sb}
  \tilde {\mathfrak{B}}= ( \tilde\Omega, \tilde{ \mathcal F}, \{\tilde {\mathcal {F}}_t \}_{t\geq 0},  \tilde {\mathbb P},   (\td W(t))_{t \geq 0}  ), 
\end{equation}
  and  $(\td W (t) )_{t \geq 0}$ is a Brownian motion adapted to the filtration $\{\tilde {\mathcal {F}}_t \}_{t\geq 0}$    with values in $D(A^\delta)$ and covariance operator $\mathcal{L}$,  such that 
$$ (\tilde {\uu} (\cdot),  \tilde P (\cdot), \tilde {\dd}(\cdot)          ): \tilde  \Omega \times [0, \infty) \rightarrow  H\times \R\times H^1(\T^3;\R^3)  $$
 is a  $  \tilde {\mathcal {F}}_t  $  adapted process,  and 
$$\tilde \omega\in  \tilde \Omega\to (\tilde {\uu}(\tilde \omega), \tilde P(\tilde \omega),\tilde  {\dd}(\tilde \omega))\in L^2(0, T; H)\times L^{5/3}(Q_T)\times L^2(0, T; H^1(\T^3;\R^3))$$
  is a measurable mapping,   and  there exists a set $\tilde{\Omega}_0\subset \tilde {\Omega}$ of full probability  such that on this set 
 the triple  $(\tilde {\uu}(\cdot , \tilde \omega), \tilde P(\cdot , \tilde \omega), \tilde {\dd}(\cdot ,  \tilde \omega))$ is a suitable weak solution in the sense of Definition \ref{def:suitweak}.
\label{def:marting}
\end{definition}
\subsection{Main results: existence and partial regularity}
\label{sec:main}
To state the main result of our article  we introduce the notion of singular points for a suitable weak solution.
\begin{definition}
  We call a point $z=(x,t)\in \T^3\times(0, \infty)$   \emph{regular} if there exists a neighborhood of $z$ where $({\utwo}, \nabla {\dd})$ is essentially bounded. The other points will be called \emph{singular}. The set of singular points will be denoted by $\Sigma$.
\end{definition}

Our main result is as follows. 
\begin{theorem}\label{thm:main2}With the assumptions \eqref{eqn:assum},  
suppose furthermore that the Brownian motion $\left (  W (t) \right)_{t \geq 0}$ takes value in
 $D(A^{\frac{1}{4}+\delta})$ for some $\delta>0$. 
Then there exists a martingale suitable weak solution  to \eqref{eqn:SEL0} globally in time such that with full probability, 
  $$\mathcal{P}^1(\Sigma)=0,$$ 
  where $\mathcal{P}^k$, $0\le k\le 4$,  denotes the $k$-dimensional Hausdorff measure on $\T^3\times \R_+$ with respect to the parabolic distance 
  $\delta_p( (x, t), (y, s)):=\max\{|x-y|, \sqrt{|t-s|}\},  \forall (x, t), (y, s)\in \T^3\times\R_+. $
\end{theorem}

\section{Existence of martingale suitable weak solutions}
\label{sec:exis}
In this section, we will prove  the first part of the conclusion of Theorem \ref{thm:main2}, that is, 
  existence of martingale suitable weak solutions to \eqref{eqn:SEL0} in the sense of Definition \ref{def:mar}.
\subsection{Energy-inequality-preserving  approximation}

We begin the proof of existence of 
martingale suitable weak solutions by   introducing a system that converges to the SEL while preserving relevant energy inequalities. 

Given ${\uone}(\omega)$ as in Lemma \ref{lemma:boundz}, we introduce the following 
%
{approximating equations}
  to \eqref{eqn:SELu2}: 
\begin{equation}
\begin{cases}
  &  \pa_t {\utwo}^\sigma-\Delta {\utwo}^\sigma+({\uone}+\Phi_\sigma[{\utwo}^\sigma])\cdot \nabla({\uone}+{\utwo}^\sigma)+\nabla {\ptwo}^\sigma\\
&=-\nabla \Phi_\sigma[{\dd}^
    \sigma]\cdot (\Delta {\dd}^\sigma-f({\dd}^\sigma)), \\
  &  \nabla\cdot {\utwo}^\sigma=0, \\
  &  \pa_t {\dd}^\sigma-({\uone}+{\utwo}^\sigma)\cdot \nabla \Phi_\sigma[{\dd}^\sigma]=\Delta {\dd}^\sigma-f({\dd}^\sigma),
\end{cases}
  \label{eqn:appro}
\end{equation}
where $\sigma>0$ is the approximation parameter and $\Phi_\sigma[f]$ denotes the
smooth mollification of a function $f$. More precisely, suppose that 
 {$\psi\in C_0^\infty(\R^3)$}
is a standard mollifier 
{supported in    $B_1(0)$},
and let $\psi_\sigma:=\sigma^{-3}\psi(x/\sigma)$. Then we define $\Phi_\sigma[f]$ via the following convolution
$$\Phi_\sigma[f]=\int_{\R^3}\psi_\sigma(x-y)f(y)\rd y,$$ where we extend $f$ periodically to the whole $\R^3$. It is easy to show that 
\begin{equation*}
	\|\Phi_{\sigma}[f]\|_{L^p(\T^3)}\le \|f\|_{L^p(\T^3)}, \qquad \lim_{\sigma\to 0}\|\Phi_{\sigma}[f]-f\|_{L^p(\T^3)}=0.
\end{equation*}
As the key feature of this type of approximations is that   global and local energy inequalities are preserved,  it is also  
  applicable to more complicated hydrodynamic equations (see,  for instance, \cite{Du2020,  du2021}). 
 {Moreover, we would like to point out that, for the Ericksen stress tensor term, we have employed the same observation as discussed in  Remark \ref{rem1}.}

We note that \eqref{eqn:appro} has
better regularity in nonlinear terms, and one could obtain the
existence of solutions  {globally strong in the PDE sense} via a similar argument using the Galerkin method and Banach fixed point theorem as in \cite{Flandoli2002} and \cite{Romito2010}. 
 {More precisely, for each fixed $\sigma$, we have
\begin{align} \label{eq:smooth}
\begin{cases}
&{\utwo}^\sigma \in L^2 (0, T; H^2 (\T^3)) \cap L^ \infty (0, T; V), \quad  \pa_t {\utwo}^\sigma \in L^2 (0, T; H),\\
 &{{\dd}}^\sigma \in L^2 (0, T; H^2 (\T^3)) \cap L^ \infty (0, T; V), \quad  \pa_t {{\dd}}^\sigma \in L^2 (0, T; H)  ,  \\
&{\ptwo}^\sigma \in L^2 (0, T;   H^1 (\T^3)). 
\end{cases}
\end{align}
%
Moreover, we note that 
{subtracting any constant from $\pi^\sigma$  yields again a solution to \eqref{eqn:appro}}. 
%
%
%
So
 by subtracting    ${\ptwo}^\sigma $ with its mean, $\int_{\T^3}{\ptwo}^\sigma $, we will obtain a function  with zero mean; we will choose this function as the solution. Thus,  without loss of generality,  from now on we will  assume that 
\begin{equation}\label{eq:mean0}
\int_{\T^3}{\ptwo}^\sigma =0. 
\end{equation}
}
 

\subsubsection{Pathwise estimates independent of the approximation parameter}
{We  can derive the following estimates  on $ ({\utwo}^\sigma,    {\ptwo}^\sigma,  {\dd}^\sigma )$ independent of $\sigma$,  {with a fixed   arbitrary   $\omega\in \Omega$ such that $\uone(\omega)$ satisfies
 {the bound specified in} 
 \eqref{eqn:zregu}}; these estimates will be used to achieve the compactness argument later on (see Section \ref{sec:comp}).}
\begin{lemma}\label{lem:estimates}
{[Energy estimates independent of $\sigma$]}
 Fix $\omega \in \Omega$ such that  ${\uone}(\omega)\in  L_{\rm loc}^\infty(\mathbb{T}^3\times[0,\infty))$. Consider the triple $({\utwo}^\sigma(\omega), {\ptwo}^\sigma(\omega),  {\dd}^\sigma(\omega))$,  a classical strong (in the PDE sense) solution   to \eqref{eqn:appro}, then we have the following estimates:
\begin{equation}\displaystyle
 \begin{cases}
&
{\|{\utwo}^\sigma \|_{L ^\infty (0, T; H)}^2+\|{\dd}^\sigma \|_{L^ \infty (0, T; H^1(\T^3))}^2 
+ \|\nabla {\utwo}^\sigma \|_{L^2(Q_T )}^2+\|\nabla^2 {\dd}^\sigma \|_{L^2(Q_T)}^2  
}
 \\
&\le   \Psi(T), \\
  & \|{\ptwo}^\sigma\|_{L^{\frac{5}{3}}(Q_T)}  
  \le C\Psi(T) +  {C   \|{\uone}\|_{L^{\infty}(Q_T)}^2}   ,   \\
  &\|\pa_t {\utwo}^\sigma\|_{L^2(0, T; D(A^{-1}))+L^2(0, T; H^{-1}(\T^3))+L^{\frac{5}{4}}(Q_T)} 
\le    C(
  \sqrt{\Psi(T)}
  +\Psi(T)), \\
   & \|\pa_t \dd^\sigma\|_{L^{\frac{5}{3}}(Q_T)} 
    \le C(\Psi(T)+T^{\frac{1}{10}}
    \sqrt{\Psi(T)}),    \\
 & |\Delta {\dd}^\sigma-f({\dd}^\sigma)|^2  _{L^2(Q_T )} \leq C\Psi(T)   , 
%
\end{cases}  \label{eqn:lemmaa}
\end{equation}
where $C$ is a constant that is independent of $\sigma$, and 
 \begin{equation}
\begin{aligned}
\Psi(T)  =  & C\left(\|\uu_0\|_{H}^2+\|\dd_0\|_{H^1(\T^3)}^2+ \|{\uone}(t)\|_{L^4(Q_T)}^4  \right) e^{C\int_0^T  {\frac{1}{2}}+ \|\uone(t)\|_{L^4(\T^3)}^4 \rd t}. 
 \label{eqn:GronEstp}
\end{aligned}
\end{equation}
 \end{lemma} 
%
\begin{proof}
{
The proof of this lemma is based on deriving suitable energy estimates. Due to the length of the derivation of the global energy equality, we have placed it in Appendix~\ref{secA2}. In particular, all of the estimates stated above will be derived from equation~\eqref{eqn:Engeqp} in Corollary~\ref{remark:energy}.}

{Firstly, we deal with the  two terms on the right-hand-side (RHS) of \eqref{eqn:Engeqp} one by one.}
For the first term we have the following estimation:
\begin{align*}
  &\left|\int_{\T^3}[({\uone}+\Phi_\sigma[{\utwo}^\sigma])\cdot\nabla {\utwo}^\sigma]\cdot {\uone} {\rd x}\right|\\
  \le&\int_{\T^3}|{\uone}| |\nabla {\utwo}^\sigma||{\uone}|+|\Phi_\sigma[{\utwo}^\sigma]||\nabla {\utwo}^\sigma||{\uone}|\\
 \le & \|{\uone}\|_{L^4(\T^3)}\|{\uone}\|_{L^4(\T^3)}\|\nabla {\utwo}^\sigma\|_{L^2(\T^3)}+\|\Phi_\sigma[{\utwo}^\sigma]\|_{L^4(\T^3)}\|\nabla {\utwo}^\sigma\|_{L^2(\T^3)}\|{\uone}\|_{L^4(\T^3)}\\
 \le & \frac{1}{4}\|\nabla {\utwo}^\sigma\|_{L^2(\T^3)}^2 +C\left( \|{\uone}\|_{L^4(\T^3)}^4+\|{\uone}\|_{L^4(\T^3)}^2 \|{\utwo}^\sigma\|_{L^4(\T^3)}^2 \right)\\
\le  &  \frac{1}{4}\|\nabla {\utwo}^\sigma\|_{L^2(\T^3)}^2 + C \left( \|{\uone}\|_{L^4(\T^3)}^4+  \|{\uone}\|_{L^4(\T^3)}^2 \| {\utwo}^\sigma\|_{  {H^1(\T^3)}}^{\frac{3}{2}}\|{\utwo}^\sigma\|_{L^2(\T^3)}^{\frac{1}{2}}\right)\\
\le  &  \frac{1}{4}\|\nabla {\utwo}^\sigma\|_{L^2(\T^3)}^2 + C \left( \|{\uone}\|_{L^4(\T^3)}^4+  \|{\uone}\|_{L^4(\T^3)}^2 \|\nabla {\utwo}^\sigma\|_{L^2(\T^3)}^{\frac{3}{2}}\|{\utwo}^\sigma\|_{L^2(\T^3)}^{\frac{1}{2}}
+  { \|{\uone}\|_{L^4(\T^3)}^2 \| {\utwo}^\sigma\|_{L^2(\T^3)}^{2} }
\right)
\\
\le  & \frac{1}{2} \|\nabla {\utwo}^\sigma\|_{L^2(\T^3)}^2 +C\left( \|{\uone}\|_{L^4(\T^3)}^4+ \|{\uone}\|_{L^4(\T^4)}^4 \| {\utwo}^\sigma\|_{L^2(\T^3)}^2 
+  { \left (\dfrac{1}{2}\|{\uone}\|_{L^4(\T^3)}^4 + \dfrac{1}{2} \right)  \| {\utwo}^\sigma\|_{L^2(\T^3)}^{2} }
\right)
\\
\le  
& \frac{1}{2} \|\nabla {\utwo}^\sigma\|_{L^2(\T^3)}^2 +C\left( \|{\uone}\|_{L^4(\T^3)}^4+
 {\dfrac{3}{2}} \|{\uone}\|_{L^4(\T^4)}^4 \| {\utwo}^\sigma\|_{L^2(\T^3)}^2 
+  {\dfrac{1}{2} 
 \| {\utwo}^\sigma\|_{L^2(\T^3)}^{2} 
}
\right), 
  \end{align*}
where $C$ is a constant\footnote{Note that C may be different at each 
%
{occurrence}
below but all of them are independent of $\sigma$.} which is  independent of $\sigma$. 
For the second term, by the H\"{o}lder inequality and Young inequality,  we obtain the following estimation:
\begin{align*}
  &\left|\int_{\T^3}{\uone}\cdot \nabla 
{
\Phi_{\sigma}[{\dd}^\sigma]
}
\cdot (\Delta {\dd}^\sigma-f({\dd}^\sigma)){\rd x}\right|\\
\le   &\frac{1}{8}\|\Delta {\dd}^\sigma-f({\dd}^\sigma)\|_{L^2(\T^3)}^2+C \|{\uone}\|_{L^4(\T^3)}^2 \|\nabla {\dd}^\sigma\|_{L^4(\T^3)}^2, \\
\le   &\frac{1}{8}\|\Delta {\dd}^\sigma-f({\dd}^\sigma)\|_{L^2(\T^3)}^2+C\|{\uone}\|_{L^4(\T^3)}^2 \|\nabla^2 {\dd}^\sigma\|_{L^2(\T^3)}^{\frac{3}{2}}\|{\utwo}^\sigma\|_{L^2(\T^3)}^{\frac{1}{2}}\\
\le  & \frac{1}{8}\|\Delta {\dd}^\sigma-f({\dd}^\sigma)\|_{L^2(\T^3)}^2+\frac{1}{4}\|\nabla^2 {\dd}^\sigma\|_{L^2(\T^3)}^2+ C \|{\uone}\|_{L^4(\T^3)}^{4}\|{\utwo}^\sigma\|_{L^2(\T^3)}^2. 
\end{align*}
 {Combining these two estimations, together with \eqref{eqn:Engeqp},   we obtain}
\begin{equation}
  \begin{aligned}
    &\frac{\rd}{{\rd t}}\int_{\T^3}\frac{1}{2}(|{\utwo}^\sigma|^2+|\nabla {\dd}^\sigma|^2)+F({\dd}^\sigma) {\rd}x+
\dfrac{1}{2}\int_{\T^3} |\nabla {\utwo}^\sigma|^2 {\rd}x\\
&
+\dfrac{7}{8}\int_{\T^3}|\Delta {\dd}^\sigma-f({\dd}^\sigma)|^2{\rd}x
\\
    \leq
&
 C\left( \|{\uone}\|_{L^4(\T^3)}^4+
 {\dfrac{3}{2}} \|{\uone}\|_{L^4(\T^4)}^4 \| {\utwo}^\sigma\|_{L^2(\T^3)}^2 
+  {\dfrac{1}{2} 
 \| {\utwo}^\sigma\|_{L^2(\T^3)}^{2} 
}
\right)
\\
&
 +\frac{1}{4}\|\nabla^2 {\dd}^\sigma\|_{L^2(\T^3)}^2+ C \|{\uone}\|_{L^4(\T^3)}^{4}\|{\utwo}^\sigma\|_{L^2(\T^3)}^2
\\
\leq
&
 C\left( \|{\uone}\|_{L^4(\T^3)}^4+
   {\dfrac{1}{2} 
 \| {\utwo}^\sigma\|_{L^2(\T^3)}^{2} 
}
\right)
 +\frac{1}{4}\|\nabla^2 {\dd}^\sigma\|_{L^2(\T^3)}^2+ C \|{\uone}\|_{L^4(\T^3)}^{4}\|{\utwo}^\sigma\|_{L^2(\T^3)}^2. 
  \end{aligned}
  \label{eqn:Engeqp1}
\end{equation}
Here we note again that the constant C may be different at each 
{occurrence}. 
%
%
%
%
%
%
%
 {Now we deal with the last term on the left-hand side (LHS) of (\ref{eqn:Engeqp1}) as follows}
  \begin{align*}
  &  \int_{\T^3}|\Delta {\dd}^\sigma-f({\dd}^\sigma)|^2 {\rd x}\\
=&\int_{\T^3}[|\Delta {\dd}^\sigma|^2+|f({\dd}^\sigma)|^2-2\Delta {\dd}^\sigma\cdot f({\dd}^\sigma)]{\rd x}\\
 =   &\int_{\T^3}[|\Delta {\dd}^\sigma|^2+|f({\dd}^\sigma)|^2+2 \nabla {\dd}^\sigma: \nabla f({\dd}^\sigma)]{\rd x}\\
  =  &\int_{\T^3}[|\Delta {\dd}^\sigma|^2+|f({\dd}^\sigma)|^2-2|\nabla {\dd}^\sigma|^2+2|\nabla {\dd}^\sigma|^2|{\dd}^\sigma|^2+4|(\nabla {\dd}^\sigma)^{\top} {\dd}^\sigma|^2]{\rd x}.
  \end{align*}
 {By slightly rearranging the terms on  the RHS, we arrive at}
\begin{equation}
\label{eq:rea}
  \begin{aligned}
  &  \int_{\T^3}|\Delta {\dd}^\sigma-f({\dd}^\sigma)|^2 {\rd x}\\
  =  &\int_{\T^3}[|\Delta {\dd}^\sigma|^2+|f({\dd}^\sigma)|^2+2|\nabla {\dd}^\sigma|^2|{\dd}^\sigma|^2+4|(\nabla {\dd}^\sigma)^{\top} {\dd}^\sigma|^2]{\rd x}
 -2  \int_{\T^3}|\nabla {\dd}^\sigma|^2 {\rd x}.\\
  \end{aligned}
\end{equation}
 {Combining  \eqref{eq:rea} and \eqref{eqn:Engeqp1}, we obtain, after some slight rearrangements:}
\begin{equation}  \label{eqn:Engeq}
  \begin{aligned}
    &\frac{\rd}{{\rd t}}\int_{\T^3}\frac{1}{2}(|{\utwo}^\sigma|^2+|\nabla {\dd}^\sigma|^2)+F({\dd}^\sigma) {\rd}x+|\nabla {\utwo}^\sigma|^2_{L^2 (\T^3)}
+\dfrac{1}{8}\int_{\T^3}|\Delta {\dd}^\sigma-f({\dd}^\sigma)|^2 {\rd x}
\\
&
+\dfrac{6}{8}\int_{\T^3}[|\Delta {\dd}^\sigma|^2+|f({\dd}^\sigma)|^2 +2|\nabla {\dd}^\sigma|^2|{\dd}^\sigma|^2+4|(\nabla {\dd}^\sigma)^{\top} {\dd}^\sigma|^2]{\rd x}
\\
   &\leq 
 C\left( \|{\uone}\|_{L^4(\T^3)}^4
+  {\dfrac{1}{2} 
 \| {\utwo}^\sigma\|_{L^2(\T^3)}^{2} 
}
\right)
 +\frac{1}{4}\|\nabla^2 {\dd}^\sigma\|_{L^2(\T^3)}^2+ C \|{\uone}\|_{L^4(\T^3)}^{4}\|{\utwo}^\sigma\|_{L^2(\T^3)}^2 
\\
&+ \dfrac{6}{4}\int_{\T^3} |\nabla {\dd}^\sigma|^2 {\rd x}.
  \end{aligned}
\end{equation}
{
By employing techniques including an application of the Gronwall  inequality to  \eqref{eqn:Engeq} (for a detailed derivation, we refer the reader to Section~\ref{secA3} of the Appendix), we obtain}
\begin{equation}
\begin{aligned}\displaystyle
&\sup_{0<t<T}\left(\|{\utwo}^\sigma(t)\|_{H}^2+\|{\dd}^\sigma(t)\|_{H^1(\T^3)}^2\right)\\
&
+\int_{0}^{T} \|\nabla {\utwo}^\sigma(t)\|_{L^2(\T^3)}^2+\|\nabla^2 {\dd}^\sigma(t)\|_{L^2(\T^3)}^2  {\rd t}\\
 \le  & C\left(\|\uu_0\|_{H}^2+\|\dd_0\|_{H^1(\T^3)}^2+\int_0^T\|{\uone}(t)\|_{L^4(\T^4)}^4\rd t\right) e^{C\int_0^T  {\frac{1}{2}}+ \|\uone(t)\|_{L^4(\T^3)}^4 \rd t}\\
=& C\left(\|\uu_0\|_{H}^2+\|\dd_0\|_{H^1(\T^3)}^2+ \|{\uone}(t)\|_{L^4(Q_T)}^4  \right) e^{C\int_0^T  {\frac{1}{2}}+ \|\uone(t)\|_{L^4(\T^3)}^4 \rd t}\\
 =:&\Psi(T).
\label{eqn:GronEst}
\end{aligned}
\end{equation}
Then from (\ref{eqn:GronEst})   we infer  \eqref{eqn:GronEstp}, and also we obtain
\begin{equation*}
  \sup_{0<t<T}\left(\|{\utwo}^\sigma(t)\|_{H}^2+\|{\dd}^\sigma(t)\|_{H^1(\T^3)}^2\right)
+ \|\nabla {\utwo}^\sigma \|_{L^2(Q_T )}^2+\|\nabla^2 {\dd}^\sigma \|_{L^2(Q_T)}^2  
\leq \Psi(T), 
\end{equation*}
which implies  (\ref{eqn:lemmaa})$_1$. 
Also   from   (\ref{eqn:Engeq})  and \eqref{eqn:GronEst} we obtain  (\ref{eqn:lemmaa})$_5$.

Next, we aim to derive  (\ref{eqn:lemmaa})$_2$, a  uniform   bound for  ${\ptwo}^\sigma$ in $\sigma$.
Taking the divergence of \eqref{eqn:appro}$_1$ we obtain
\begin{equation}
  -\Delta {\ptwo}^\sigma=\div^2[({\uone}+\Phi_\sigma[{\utwo}^\sigma])\otimes({\uone}+{\utwo}^\sigma)]+\div(\nabla \Phi_\sigma[{\dd}^\sigma]\cdot (\Delta {\dd}^\sigma-f({\dd}^\sigma))).
  \label{eqn:approPeq}
\end{equation}
Applying the elliptic theory  {as in Lemma 5.1 of \cite{Robinson2016NS}},  and utilizing   H\"older's inequality we obtain:  
\begin{equation}\label{eq:rob}
\begin{aligned}
   \|{\ptwo}^\sigma\|_{L^{\frac{5}{3}}(Q_T)} 
\le& C\|({\uone}+\Phi_\sigma[{\utwo}^\sigma])\otimes ({\uone}+{\utwo}^\sigma)\|_{L^{\frac{5}{3}}(Q_T)}\\
&+\|\nabla (\Phi_\sigma[{\dd}^\sigma])\cdot (\Delta {\dd}^\sigma-f({\dd}^\sigma))\|_{ L^{\frac{5}{3}} (0, T; W^{-1, \frac{5}{3}} (\T^3))}\\
\le&C \left(\|{\uone}\|_{L^{\frac{10}{3}}(Q_T)}^2+\|{\utwo}^\sigma\|^2_{L^{\frac{10}{3}}(Q_T)}\right)\\
&+\|\nabla (\Phi_\sigma[{\dd}^\sigma])\cdot (\Delta {\dd}^\sigma-f({\dd}^\sigma))\|_{ L^{\frac{5}{3}} (0, T; W^{-1, \frac{5}{3}} (\T^3))}.
\end{aligned}\end{equation}
Now we first deal with the first term on the RHS. We   easily observe that 
\begin{equation}\label{eq:z1}
\|{\uone}\|_{L^{\frac{10}{3}}(Q_T)} \leq \|{\uone}\|_{L^{\infty}(Q_T)} .
\end{equation}  
By the Riesz–Thorin theorem
and  the Sobolev embedding theorem in 3D we observe that 
\begin{equation}
\begin{aligned}\label{eq:sob}
\|{\utwo}^\sigma\| _{L^{\frac{10}{3}}(Q_T)}& \leq C \|{\utwo}^\sigma\| _{L^{2} (0, T; L^6 (\T^3)) }  ^{\frac{2}{5}}\|{\utwo}^\sigma\| _{L^{\infty} (0, T; L^2 (\T^3)) } ^{\frac{3}{5}}\\
&\leq C \|{\utwo}^\sigma\| _{L^{2} (0, T; H^1 (\T^3)) }  ^{\frac{2}{5}}\|{\utwo}^\sigma\| _{L^{\infty} (0, T; L^2 (\T^3)) } ^{\frac{3}{5}}
\end{aligned}
\end{equation}
Then this together with \eqref{eq:z1} and  \eqref{eq:rob} imply that 
\begin{equation}\label{eq:rob7}
\begin{aligned}
   \|{\ptwo}^\sigma\|_{L^{\frac{5}{3}}(Q_T)} 
\le&C \left(\|{\uone}\|_{L^{\infty}(Q_T)}^2+ C \|{\utwo}^\sigma\| _{L^{2} (0, T; H^1 (\T^3)) }  ^{\frac{4}{5}}\|{\utwo}^\sigma\| _{L^{\infty} (0, T; L^2 (\T^3)) } ^{\frac{6}{5}}\right)\\
&+\|\nabla (\Phi_\sigma[{\dd}^\sigma])\cdot (\Delta {\dd}^\sigma-f({\dd}^\sigma))\|_{ L^{\frac{5}{3}} (0, T; W^{-1, \frac{5}{3}} (\T^3))}.
\end{aligned}\end{equation}
Next we deal with the second term  on the RHS of \eqref{eq:rob7}.  We first observe that in 3D, the space $L^{\frac{15}{14}} (\T^3)$ is   embedded in $W^{-1, \frac{5}{3}} (\T^3)$ thanks to the Sobolev inequalities , so applying this observation to the second term on the RHS   we obtain
\begin{equation}\label{eq:rob2}
\begin{aligned}
   \|{\ptwo}^\sigma\|_{L^{\frac{5}{3}}(Q_T)} 
\le& C \left(\|{\uone}\|_{L^{\infty}(Q_T)}^2+ C \|{\utwo}^\sigma\| _{L^{2} (0, T; H^1 (\T^3)) }  ^{\frac{4}{5}}\|{\utwo}^\sigma\| _{L^{\infty} (0, T; L^2 (\T^3)) } ^{\frac{6}{5}}\right)
\\
&+\|\nabla (\Phi_\sigma[{\dd}^\sigma])\cdot (\Delta {\dd}^\sigma-f({\dd}^\sigma))\|_{ L^{\frac{5}{3}} (0, T;  L^{\frac{15}{14}} (\T^3))}.
\end{aligned}\end{equation}
Now applying   H\"older's inequality to the second term on the RHS   we obtain 
\begin{equation}\label{eq:rob3}
\begin{aligned}
   \|{\ptwo}^\sigma\|_{L^{\frac{5}{3}}(Q_T)} 
\le& C \left(\|{\uone}\|_{L^{\infty}(Q_T)}^2+ C \|{\utwo}^\sigma\| _{L^{2} (0, T; H^1 (\T^3)) }  ^{\frac{4}{5}}\|{\utwo}^\sigma\| _{L^{\infty} (0, T; L^2 (\T^3)) } ^{\frac{6}{5}}\right)
\\
&+C\|\nabla {\dd}^\sigma\|^2_{L ^{10} (0, T;  L^{\frac{30}{13}} (\T^3)}\|\Delta {\dd}^\sigma-f({\dd}^\sigma)\|_{L^2(Q_T)}.
\end{aligned}\end{equation}
Now  applying Riesz–Thorin theorem  to the second term on the RHS of (\ref{eq:rob3}), and utilizing  the    Sobolev embedding theorem in 3D, we observe that 
\begin{equation}\label{eq:rob5}
\begin{aligned}
\|\nabla {\dd}^\sigma\|_{L ^{10} (0, T;  L^{\frac{30}{13}} (\T^3))} \leq C \|\nabla {\dd}^\sigma\|_{L ^{\infty} (0, T;  L^{2} (\T^3))}^{\frac{1}{5}} \times  \|\nabla {\dd}^\sigma\|_{L ^{2} (0, T;  L^{6} (\T^3))}^{\frac{4}{5}}
\\
\leq
 C \|\nabla {\dd}^\sigma\|_{L ^{\infty} (0, T;  L^{2} (\T^3))}^{\frac{1}{5}} \times  \|\nabla {\dd}^\sigma\|_{L ^{2} (0, T;  H^1 (\T^3))}^{\frac{4}{5}}
\end{aligned}
\end{equation}
This together with (\ref{eq:rob3}) implies that 
\begin{equation}\label{eq:rob6}
\begin{aligned}
  & \|{\ptwo}^\sigma\|_{L^{\frac{5}{3}}(Q_T)} \\
\le& C  \left(\|{\uone}\|_{L^{\infty}(Q_T)}^2+ C \|{\utwo}^\sigma\| _{L^{2} (0, T; H^1 (\T^3)) }  ^{\frac{4}{5}}\|{\utwo}^\sigma\| _{L^{\infty} (0, T; L^2 (\T^3)) } ^{\frac{6}{5}}\right)
\\
&+C \|\nabla {\dd}^\sigma\|_{L ^{\infty} (0, T;  L^{2} (\T^3))}^{\frac{2}{5}}   \|\nabla {\dd}^\sigma\|_{L ^{2} (0, T;  H^1 (\T^3))}^{\frac{8}{5}}
\|\Delta {\dd}^\sigma-f({\dd}^\sigma)\|_{L^2(Q_T)}\\
\end{aligned}\end{equation}
Finally from \eqref{eq:rob6},      (\ref{eqn:lemmaa})$_1$,  and (\ref{eqn:lemmaa})$_5$     we obtain       
\begin{equation}
  \|{\ptwo}^\sigma\|_{L^{\frac{5}{3}}(Q_T)}  
  \le C\Psi(T) +  {C   \|{\uone}\|_{L^{\infty}(Q_T)}^2},  \label{eqn:paiest}
\end{equation} which is \eqref{eqn:lemmaa}$_4$. 


{Next}, 
to prove \eqref{eqn:lemmaa}$_3$, 
  we can utilize \eqref{eqn:appro}$_1$ to show that 
\begin{equation}
\begin{aligned}
  &\|\pa_t {\utwo}^\sigma\|_{L^2(0, T; D(A^{-1}))+L^2(0, T; H^{-1}(\T^3))+L^{\frac{5}{4}}(Q_T)}\\
\le  & C(\|{\utwo}^\sigma\|_{L^2(Q_T)} +\|\uone\|_{L^4(Q_T)}^2+\|{\utwo}^\sigma\|_{L^4(Q_T)}^2\\
&+\|\nabla \dd^\sigma\|_{L^{\frac{10}{3}}(Q_T)}\|\Delta \dd^\sigma-f({\dd}^\sigma)\|_{L^{2}(Q_T)})\\
\le  & C(
  \sqrt{\Psi(T)}
  +\Psi(T)). 
  \label{eqn:patvEq}
  \end{aligned}
\end{equation}

Finally, to prove \eqref{eqn:lemmaa}$_4$, 
we can utilize \eqref{eqn:appro}$_3$ to derive that
\begin{equation}
  \begin{aligned}
    &\|\pa_t \dd^\sigma\|_{L^{\frac{5}{3}}(Q_T)}\\
  \le  & C  \|{\uone}\|_{L^{\frac{10}{3}}(Q_T)}\|\nabla \dd^\sigma\|_{L^{\frac{10}{3}}(Q_T)}+
C\|{\utwo}^\sigma\|_{L^{\frac{10}{3}}(Q_T)}\|\nabla \dd^\sigma\|_{L^{\frac{10}{3}}(Q_T)} 
\\
&+
CT^{\frac{1}{10}}\|\Delta \dd^\sigma-f({\dd}^\sigma)\|_{L^2(Q_T)} \\
   \le  &C(\Psi(T)+
CT^{\frac{1}{10}}
    \sqrt{\Psi(T)}).    
\end{aligned}
  \label{eqn:patdEq}
\end{equation}
\end{proof}

\subsection{Approaching the existence of 
{martingale}
 suitable solutions}
 Based on the pathwise estimates derived above,  we can infer  a compactness argument from
  tightness properties of the approximation sequence.  Next we will use 
  the  Skorokhod embedding theorem  (see Theorem 2.4 in \cite{ZabczykDaPrato1}) to construct an approximation sequence with a possible shift of the stochastic basis. It is then necessary to verify that the key energy equality \eqref{eqn:LocEngy} is preserved by the constructed sequence. Finally   
{
passing to the limit  in (\ref{eqn:appro})  
}
we  obtain
the global existence of   martingale suitable weak solutions  satisfying the local energy inequalities as desired.

\subsubsection{Compactness of the approximation sequence}
\label{sec:comp}
 Let $\nu^\sigma$ be the law of the random variable $U_{\sigma}:=({\utwo}^\sigma, {\ptwo}^\sigma, {\dd}^\sigma, {\uone}, W)$ with values in  $\mathcal{E}$:
$$\mathcal{E}=L^2(0, T; H)\times L^{\frac{5}{3}}(Q_T)\times L^2(0, T; H^1(\T^3;\R^3))\times C([0,T]; H)\times C_0([0,T]; H). $$
Let $\mu^\sigma$ be the projection of $\nu^\sigma$ in the variables $({\utwo}^\sigma, {\dd}^\sigma)$, that is,  the law of $({\utwo}^\sigma, {\dd}^\sigma)$.  
%
The key step to establish the probabilistic convergence of this approximation scheme
is to show that the family of measures $\mu^\sigma$ is tight. Namely, for each $\varepsilon>0$, there exists a compact set $K_\varepsilon$ in $\mathcal{E}$ such that 
\begin{equation}\label{eq:tight}
\mu^\sigma(K_\varepsilon)\ge 1-\varepsilon, \text{ for }0<\sigma<1.
\end{equation}
We take 
$$
\begin{aligned}
  K_R&:=\Big\{({\utwo}, \dd)\Big| \|{\utwo}\|_{L^\infty(0, T; H)\cap L^2(0, T; V)}+\|\pa_t {\utwo}\|_{L^2(0,T; D(A^{-1}))+L^2(0, T; H^{-1}(\T^3))+L^{\frac{5}{4}}(Q_T)}\\
  &+\|{\dd}\|_{L^\infty(0, T; H^1(\T^3))\cap L^2(0, T; H^2)\cap H^1(0, T; L^2(\T^3))}+\|\pa_t \dd^\sigma\|_{L^{\frac{5}{3}}(Q_T)}\le R\Big\}.
\end{aligned}
$$
By the Aubin-Lions lemma, we conclude that $K_R$ is compact in $L^{p_1}(Q_T)\times L^{p_2}(Q_T)$ with $1<p_1<\frac{10}{3}$, and $1<p_2<10$. Moreover, by the estimates \eqref{eqn:zestima}, \eqref{eqn:GronEst}, \eqref{eqn:patvEq} and \eqref{eqn:patdEq}, we can follow the same argument as in \cite[(381)-(383)]{FlandoliGatarek1} to conclude that there exists a sufficiently large $R>0$ such that $K_{1/R}$ satisfies \eqref{eq:tight}.
Consequently,   the laws of   $({\utwo}^\sigma,  {\dd}^\sigma, {\uone}, W)$,   which will be denoted as $\xi^\sigma$ 
subsequently, 
   are tight over the phase space  $\mathcal{G}$:
$$ \mathcal{G} =L^{p_1}(Q_T)\times  L^{p_2}(Q_T)\times C([0,T]; H)\times C_0([0,T]; H),
%
%
 1<p_1<\frac{10}{3}, 1<p_2<10. $$  Since the family of measures $\{\xi^\sigma \} $ is weakly compact on $\mathcal G$,  we deduce that $\{\xi^\sigma \} $ converges weakly to a probability measure $\xi$ on $\mathcal G$ up to a subsequence. 
%
%
%
%
Hence,  using a version of the Skorokhod  embedding theorem (See \cite[Lemma 3.3]{Romito2010}), 
we can conclude that there exists a probability
triple
$  ( \tilde\Omega, \tilde{ \mathcal F},    \tilde {\mathbb P})$,
a sequence of random variables 
$\tilde{U}_k=(\tilde{{\utwo}}^{\sigma_k},\tilde{\ptwo}^{\sigma_k}, \tilde{{\dd}}^{\sigma_k}, \tilde{{\uone}}^k,  \tilde{W}^{k})$ with values in $ \mathcal{E}$,   and an element $\tilde{U} = (\tilde {\utwo},  \tilde {\ptwo}, \tilde{\dd},\tilde {\uone}, \tilde W) \in \mathcal{E}$ such that for each $k$, $\tilde{U}_k$ has the same law with $U_{\sigma_k}$, 
in particular, $\tilde{W}^{k}$   is a Brownian motion adapted to the filtration $\{\tilde {\mathcal {F}}^{k}_t \}_{t\geq 0}$ which is given by the completion of the $\sigma$-algebra generated by $\{\tilde{U}_k (s);  s\leq t\}$, and
\begin{equation}\label{eq:convergemain}
(\tilde{{\utwo}}^{\sigma_k},\tilde{{\dd}}^{\sigma_k}, \tilde{{\uone}}^k, \tilde{W}^{k}) \rightarrow 
  (\tilde {\utwo},   \tilde{\dd},\tilde {\uone}, \tilde W) \quad \tilde{\mathbb P}- a.s.,  \mbox{ as }k\to \infty,  
\end{equation}
in the topology of $\mathcal G$.  Additionally,   setting $\{\tilde {\mathcal {F}}_t \}_{t\geq 0}$ as the completion of the $\sigma$-algebra generated by $\{\tilde{U} (s);  s\leq t\}$, we obtain a new stochastic basis 
%
{
$  \tilde {\mathfrak{B}}= \left ( \tilde\Omega, \tilde{ \mathcal F}, \{\tilde {\mathcal {F}}_t \}_{t\geq 0},  \tilde {\mathbb P},   \left(\td W(t) \right)_{t \geq 0} \right)$, 
}
 as in \eqref{eq:sb}.

\subsubsection{Preservation of the key energy equality}

Next we   show that for each $k$,  the random variable $(\tilde {\utwo}^{\sigma_k},   \tilde {{\ptwo}}^{\sigma_k},  \tilde{\dd}^{\sigma_k}, 
\tilde{{\uone}}^k
)$ satisfies the local energy equality \eqref{eqn:LocEngy}. 
%
We adapt 
a 
technique
due to
Bensoussan \cite{bensoussan1995stochastic}. 
Given
%
{$\varphi\in C_0^\infty(Q_T;\R_+)$},  
 for $( {\utwo} ^\sigma,   {\ptwo}^\sigma,   {\dd}^\sigma,  {{\uone}} )$,   the components of $  {U}_\sigma$,   
define the random variable $X^\sigma$ on $(\Omega, \mathcal{F})$ as
\begin{align}\label{eq:xsigma}
    X^\sigma&:=\left|X_1^\sigma -X_2^\sigma  \right|,
\end{align}
where 
\begin{align*}
   X_1^\sigma&:= \int_{\T^3\times\{T\}}\left( \frac{|{\utwo}^\sigma|^2}{2}+\frac{|\nabla {\dd}^\sigma|^2}{2}+F({\dd}^\sigma) \right)\varphi+\int_{Q_T}(|\nabla {\utwo}^\sigma|^2+|\Delta {\dd}^\sigma|^2+|f({\dd}^\sigma)|^2)\varphi
\end{align*}
and
\begin{align*}
    X_2^\sigma&:= \int_{Q_T}\left( \frac{|{\utwo}^\sigma|^2}{2}+\frac{|\nabla {\dd}^\sigma|^2}{2}+F({\dd}^\sigma) \right)\pa_t \varphi
    +\int_{\T^3}[({\uone}+\Phi_\sigma[{\utwo}^\sigma])\cdot \nabla {\utwo}^\sigma]\cdot {\uone}\varphi\\
    &+\int_{Q_T}\left (\frac{|{\utwo}^\sigma|^2}{2} + {\utwo}^\sigma \cdot \uone\right)({\uone}+\Phi_\sigma[{\utwo}^\sigma])\cdot \nabla \varphi
    +{\ptwo}^\sigma {\utwo}^\sigma \cdot  \nabla \varphi
    +\left( \frac{|{\utwo}^\sigma|^2}{2}
    +\frac{|\nabla {\dd}^\sigma|^2}{2} \right)\Delta \varphi\\
    &+\int_{Q_T}\left( \nabla {\dd}^\sigma\odot \nabla {\dd}^\sigma-\frac{1}{2}|\nabla {\dd}^\sigma|^2{\rm I}_3 \right):\nabla^2 \varphi\\
    &+\int_{Q_T}({\uone}\cdot \nabla \Phi_\sigma[{\dd}^\sigma])\cdot (\Delta {\dd}^\sigma-f({\dd}^\sigma))
    +\int_{Q_T}[({\uone}+{\utwo}^\sigma)\cdot \nabla {\dd}^\sigma]\cdot (\nabla \varphi\cdot \nabla){\dd}^\sigma\\
    &-\int_{Q_T}f({\dd}^\sigma)(\nabla \varphi\cdot \nabla {\dd}^\sigma)-2\nabla f({\dd}^\sigma):\nabla {\dd}^\sigma \varphi .
\end{align*}
  From \eqref{eqn:LocEngy} in  Lemma \ref{lemma:LocalEnergy1}, we  infer that 
\begin{align}\label{eq:xsigma1}
 X_1^\sigma =X_2^\sigma   \mbox{ with Probability 1}. 
\end{align} 
This together with 
\eqref{eq:xsigma} implies that $X^\sigma=0$ with probability 1.  Thus we have $\mathbb{E}(g(U_\sigma))=0$,
where $g(U_\sigma):=X^\sigma/(1+X^\sigma).$
Let $\tilde{X}^k$ be defined analogously for   $(\tilde{{\utwo}}^{\sigma_k},  \tilde{{\ptwo}}^{\sigma_k},    \tilde{{\dd}}^{\sigma_k},
\tilde{{\uone}}^k
)$,   the components of $\tilde {U}_k$,  as a random variable on $(\tilde {\Omega}, \tilde {\mathcal F})$. 
Observing that $g(\cdot)$
is a deterministic bounded continuous function on 
a
subset of $\mathcal{E}$, we have
$$\tilde{\mathbb{E}}\left[ \frac{\tilde{X}^{k}}{1+\tilde{X}^k} \right]=\tilde{\mathbb{E}}[g(\tilde{U}_k)]=\mathbb{E}[g(U_{\sigma_k})]=\mathbb{E}\left[ \frac{X^{\sigma_k}}{1+X^{\sigma_k}} \right]=0, $$
which implies that 
\begin{equation}\label{eq:xi}
\mathbb{\tilde {P}} \left ( \tilde{X}^{ k}=0 \right )=1,
\end{equation} 
as required.
%
%
%

\subsubsection{Passage to the limit}
With a similar method as above,  we can further verify  that      for each $\sigma_k$,   $\tilde {U}_{k}$ satisfies the equations \eqref{eqn:SELu1}   and 
\eqref{eqn:appro}  in the sense of distribution    
$\tilde{\mathbb{P}}$-a.s..
Particularly,   
the random variables $(\tilde{{\utwo}}^{\sigma_k},  \tilde{{\ptwo}}^{\sigma_k},    \tilde{{\dd}}^{\sigma_k},    \tilde {{\uone}}^k)$
satisfy the approximation equation
%
 \eqref{eqn:appro}   $\tilde{\mathbb{P}}$-a.s.,  that is
\begin{equation}
\begin{cases}
   & \pa_t \tilde {\utwo}^{\sigma_k}-\Delta\tilde {\utwo}^{\sigma_k}+(\tilde{\uone}^k+\Phi_{\sigma_k}[\tilde  {\utwo}^{\sigma_k}])\cdot \nabla(\tilde{\uone}^k+\tilde{\utwo}^{\sigma_k})+\nabla\tilde {\ptwo}^{\sigma_k}\\
&\quad\quad\quad\quad \quad\quad=-\nabla  \Phi_{\sigma_k}[\tilde{\dd}^
    {\sigma_k}]\cdot (\Delta\tilde {\dd}^{\sigma_k}-f(\tilde{\dd}^{\sigma_k})), \\
  &  \nabla\cdot \tilde{\utwo}^{\sigma_k}=0, \\
 &   \pa_t \tilde {\dd}^{\sigma_k}-(\tilde{\uone}^k+\tilde {\utwo}^{\sigma_k})\cdot \nabla \Phi_{\sigma_k}[ \tilde {\dd}^{\sigma_k}]=\Delta \tilde {\dd}^{\sigma_k}-f(\tilde  {\dd}^{\sigma_k}). 
\end{cases}
  \label{eqn:approtilde}
\end{equation} 
As a result,   
{
passing  to the limit in
}
 (the weak form of) the equations
 \eqref{eqn:SELu1}  and 
 \eqref{eqn:approtilde}  as $\sigma_k\to 0$,    we conclude that $(\tilde {\utwo},   \tilde{\dd},  \tilde {\uone}, \tilde W)$ solves  \eqref{eqn:SELu1} and \eqref{eqn:SELu2} in the sense of distribution 
   (see Def. \ref{def:weaksolution})
   $\tilde{\mathbb{P}}$-a.s..   Thus there exists a distribution $\tilde {\ptwo}$ defined on the stochastic basis $\tilde{\mathfrak{B}}$
 which plays the role of the pressure field for $\tilde {\utwo}$. In fact, let $\tilde{{\ptwo}}$ solve the following Poisson equation
 $$
   -\Delta \tilde{{\ptwo}}=\div^2[(\tilde{{\uone}}+\tilde{{\utwo}})\otimes(\tilde{{\uone}}+\tilde{{\utwo}})]+\div(\nabla \tilde{\dd}\cdot (\Delta \tilde{\dd}-f(\tilde{\dd}))),
   $$
 with $\int_{{\mathbb T}^3} \tilde  {\ptwo} (t)   = 0$ for every $t$.  
Thanks to \eqref{eqn:approtilde},  for almost sure $\tilde \omega \in \tilde \Omega$,  we have 
the same estimate as in \eqref{eqn:paiest} for $\tilde{{\ptwo}}^{\sigma_k} (\tilde \omega)$,  which implies that the sequence $\tilde{{\ptwo}}^{\sigma_k}(\tilde \omega)$ is bounded uniformly in $L^{5/3} (Q_T)$ independent of $\sigma$.  Hence we have for almost every fixed $\tilde \omega \in \tilde \Omega$,  there exists a subsequence    
%
which may depend on
  $\tilde \omega$ such that
\begin{equation}\label{eq:piconv}
 \tilde{{\ptwo}}^{\sigma_k} (\tilde\omega)\to \tilde {\ptwo} (\tilde \omega) \mbox{ weakly in }L^{\frac{5}{3}}(Q_T)  \mbox{ up to the subsequence}. 
\end{equation}

Finally, we verify that $(\tilde {\utwo},  \tilde {\ptwo}, \tilde {\dd})$ satisfies the local energy inequality described in Definition \ref{def:suitweak}.   For almost sure $\tilde \omega \in \tilde \Omega$,  as proven above (see equation \eqref{eq:xi}) we know that  the random variables $(\tilde {\utwo}^{\sigma_k} (\tilde\omega),   \tilde {{\ptwo}}^{\sigma_k}(\tilde \omega),  \tilde{\dd}^{\sigma_k}(\tilde \omega), 
\tilde{{\uone}}^k (\tilde \omega))$ satisfy the local energy   equality \eqref{eqn:LocEngy}.  Thus for each such $\tilde \omega$, passing to    the limit  
upon the corresponding subsequence as in  \eqref{eq:piconv}(using also   \eqref{eq:convergemain}),    we obtain   the local energy inequality as (\ref{eq:ine1}) and (\ref{eq:ine2})  for $(\tilde {\utwo} (\tilde \omega),  \tilde {\ptwo}  (\tilde \omega),  \tilde{\dd} (\tilde \omega),
\tilde{{\uone}} (\tilde \omega))$ 
 by the lower semicontinuity.  

\section{Partial regularity of the martingale suitable weak solutions}
\label{sec:partial}

With the global existence of martingale suitable weak solutions in hand  (Section \ref{sec:exis}), we now focus on the   second result of  Theorem \ref{thm:main2},  
 that is,  to establish the partial regularity of these martingale suitable weak solutions.

\subsection{Blowing up analysis}

First, we recall the maximum principle \cite[Lemma 2.2]{lin2016global} for the $L^\infty$ of $\dd$ from the transported Ginzburg--Landau heat flow :
\begin{equation}
	\left\{
	\begin{array}{l}
		\pa_t {\dd}+{\uu}\cdot \nabla {\dd}=\Delta {\dd}-f({\dd}), \\
		\nabla \cdot  \uu=0, \\
		{\dd}(0)={\dd}_0.
	\end{array}
	\right.
	\label{eqn:transprotHE}
\end{equation}
\begin{lemma}\label{lemma:maximum}
	For $T>0$, assume that ${\uu}\in L^2(0,T; V )$ and ${\dd}_0\in H^1(\T;\R^3)$ with 
	$|{\dd}_0(x)|\le 1$ a.e. $x\in \mathbb{T}^3$. Suppose ${\dd}\in L^2(0, T; H^2(\T^3, \R^3))$ is a weak solution to \eqref{eqn:transprotHE}. Then $|{\dd}(x, t)|\le 1$ for a.e. $(x, t)\in 
 Q_T$.
\end{lemma}

The following lemma will play an essential role in the partial regularity of suitable weak solutions. 
\begin{lemma}	\label{lemma:smallenergyimprv}
	There exist $\varepsilon_0>0$, $0<\tau_0<\frac{1}{2}$ 
 and $C_0>0$ such that if $({\utwo}, {\dd}, {\ptwo})$ is a suitable weak solution in $\mathbb{T}^3\times(0, \infty)$, then for $z_0=(x_0, t_0)\in \mathbb{T}^3\times(r^2, \infty)$, and $r>0$, if it holds that 
	$$|{\uone}|\le M,\quad |{\dd}|\le M  \text{ in }\Q_r(z_0),$$ where $\Q_r(z_0)=B_r(x_0)\times(t_0-r^2, t)$,   
	and
	$$\Theta_{z_0}(r):=r^{-2}\int_{\Q_r(z_0)}(|{\utwo}|^3+|\nabla {\dd}|^3)+\left( r^{-2}\int_{\Q_r(z_0)}|\ptwo|^{\frac{3}{2}} \right)^2\le \varepsilon_0^3,$$
	then 
\begin{align*}
	\Theta_{z_0}(\tau_0 r) &=(\tau_0 r)^{-2}\int_{\Q_{\tau_0 r}(z_0)}(|{\utwo}|^3+|\nabla {\dd}|^3)+\left( (\tau_0 r)^{-2}\int_{\Q_{\tau_0 r}(z_0)}|\ptwo|^{\frac{3}{2}} \right)^2\\
&\le  \frac{1}{2}\max\{\Theta_{z_0}(r), C_0r^3\}.
\end{align*}
\end{lemma}
\begin{proof}
	Proof by contradiction. Suppose that 
 it is false. Then there exists an $M_0>0$ such that for any $\tau\in(0, \frac{1}{2})$, there exist  
   sequences
 $\varepsilon_i\to 0$, $C_i\to \infty$, $r_i>0$, and $z_i = (x_i,t_i)\in \T^3\times(r_i^2, \infty)$ such that 
	\begin{equation}
		\left\{
		\begin{array}{l}
			|{\uone}|\le M_0, ,\quad |{\dd}|\le M_0 \text{ in }\Q_{r_i}(z_i), \\
			\Theta_{z_i}(r_i)=\varepsilon_i^3, \\
			\Theta_{z_i}(\tau r_i)>\frac{1}{2}\max\{\varepsilon_i^3, C_i r_i^3\}.
		\end{array}
		\right.
		\label{eqn:blowupassum}
	\end{equation}
	First we notice that $C_i r_i^3\le 2 \Theta_{z_i}(\tau r_i)\le 2\tau^{-4} \Theta_{z_i}(r_i)=2\tau^{-4}\varepsilon_i^3$.  This implies that 
	$$r_i\le \left( \frac{2\varepsilon_i^3}{C_i \tau^4} \right)^{\frac{1}{3}}\to 0.$$
	Next, we define the rescaling sequence
	$$({{\uone}}^i, {{\utwo}}^i, {{\dd}}^i, \ptwo^i)(x, t)=(r_i {\uone}, r_i {\utwo}, {\dd}, r_i^2 {\ptwo})(x_i+r_i x, t_i+r_i^2 t).$$
	It is straightforward to verify that $({\uone}^i, {\utwo}^i, {\dd}^i, {\ptwo}^i)$ satisfies the rescaled equation 
	\begin{equation}
		\left\{
		\begin{array}{l}
			\pa_t {\utwo}^i+({\uone}^i+{\utwo}^i)\cdot \nabla ({\uone}^i+{\utwo}^i)+\nabla \ptwo^i\\
\qquad \qquad \qquad  =\Delta \utwo^i-\div\left( \nabla {\dd}_i\odot \nabla {\dd}^i-\frac{1}{2}|\nabla {\dd}^i|^2{\rm I}_3-r_i^2 F({\dd}^i){\rm I}_3 \right), \\
			\nabla\cdot  {\utwo}^i=0, \\
			\pa_t {\dd}^i+({\uone}^i+{\utwo}^i)\cdot \nabla {\dd}^i=\Delta {\dd}^i-r^2_i f({\dd}^i).
		\end{array}
		\right.
		\label{eqn:rescaledEQ}
	\end{equation}
	And \eqref{eqn:blowupassum} becomes
	\begin{equation}
		\left\{
		\begin{array}{l}
			\displaystyle \int_{\Q_1(0)}(|{\utwo}^i|^3+|\nabla {\dd}^i|^3)+\left( \int_{\Q_1(0)}|\ptwo^i|^{\frac{3}{2}} \right)^{2}=\varepsilon_i^3, \\
			\displaystyle \tau^{-2}\int_{\Q_\tau(0)}(|{\utwo}^i|^3+|\nabla {\dd}^i|^3)+\left( \tau^{-2}\int_{\Q_\tau(0)}|\ptwo^i|^{\frac{3}{2}} \right)^3>\frac{1}{2}\max\{\varepsilon_i^3, C_i r_i^3\}. 
		\end{array}
		\right.
		\label{eqn:rescalenergy}
	\end{equation}
	Then we define the blow-up sequence 
	$$(\widehat{\uone}^i, \widehat{\utwo}^i, \widehat{\dd}^i, \hP^i)=\left( \frac{{\uone}^i}{\varepsilon_i}, \frac{{\utwo}^i}{\varepsilon_i}, \frac{{\dd}^i-\bar{{\dd}^i}}{\varepsilon_i} , \frac{{\ptwo}^i}{\varepsilon_i}\right), $$
	where $\bar{{\dd}^i}={\fint}_{\Q_1(0)}{\dd}^i$, which denotes the average of ${\dd}^i$ over $\Q_1(0)$. Then it turns out that $(\widehat{\uone}^i, \widehat{\utwo}^i, \widehat{\dd}^i, \widehat{\ptwo}^i)$ solves the following equation:
	\begin{equation}
		\left\{
		\begin{array}{l}
			\pa_t \widehat{\utwo}^i+\ve_i(\widehat{\uone}^i+\widehat{\utwo}^i)\cdot \nabla(\widehat{\uone}^i+\widehat{\utwo}^i)+\nabla \widehat{\ptwo}^i\\
\quad \quad \quad 
=\Delta \widehat{\utwo}^i
  -\div\left( \varepsilon_i  \nabla \widehat{\dd}^i\odot \nabla \widehat{\dd}^i-\frac{\varepsilon_i}{2}|\nabla \widehat{\dd}^i|^2{\rm I}_3 -\frac{r_i^2}{\varepsilon_i}F({\dd}^i){\rm I}_3\right), \\
			\div \widehat{\utwo}^i=0, \\
			\pa_t \widehat{\dd}^i+\ve_i(\widehat{\uone}^i+ \widehat\utwo^i)\cdot \nabla \widehat{\dd}^i=\Delta \widehat{\dd}^i-\frac{r_i^2}{\ve_i}f(\dd^i).
		\end{array}
		\right.	\label{eqn:blowingupEQ}
	\end{equation}
	Furthermore, we have 
	\begin{equation}
\begin{aligned}
		\left\{
		\begin{array}{l}
			\int_{\Q_1(0)}\widehat{\dd}^i=0, \\
			\int_{\Q_1(0)}\left( |\widehat{\utwo}^i|^3+|\nabla\widehat{\dd}^i|^3 \right)+\left( \int_{\Q_1(0)}|\widehat{\ptwo}^i|^{\frac{3}{2}} \right)^2=1, \\
			\tau^{-2}\int_{\Q_\tau(0)}\left( |\widehat{\utwo}^i|^3+|\nabla \widehat{\dd}^i|^3 \right) +\left( \tau^{-2}\int_{\Q_\tau(0)}|\widehat{\ptwo}^i|^{\frac{3}{2}} \right)^2\\
 \quad\quad > 
{
\frac{1}{2}\max \left \{1, C_i\left( \frac{r_i}{\ve_i} \right)^3 \right\}.
}
%
%
%
		\end{array}
		\right.
		\label{eqn:blowingupEnergy}
\end{aligned}
	\end{equation}
	From \eqref{eqn:blowingupEnergy}, by weak compactness, we can show there exists
	$$(\widehat{\utwo}, \widehat{\dd}, \widehat{\ptwo})\in L^3(\Q_1(0))\times L_t^3 W_x^{1, 3}(\Q_1(0))\times L^{\frac{3}{2}}(\Q_1(0)),$$
	such that, after passing to 
the limit upon
a subsequence, we have
	$$(\widehat{\utwo}^i, \widehat{\dd}^i, \widehat{\ptwo}^i)\rightharpoonup (\widehat{\utwo}, \widehat{\dd}, \widehat{\ptwo}) \text{ in }L^3(\Q_1(0))\times L_t^3 W_x^{1, 3}(\Q_1(0))\times L^{\frac{3}{2}}(\Q_1(0)).$$
	It follows from \eqref{eqn:blowingupEnergy} and the lower semicontinuity that 
	$$\int_{\Q_1(0)}\left( |\widehat{\utwo}|^3+|\nabla\widehat\dd|^3 \right)+\left( \int_{\Q_1(0)}|\widehat{\ptwo}|^{\frac{3}{2}} \right)\le 1.$$
	We claim that 
	$$\sup_{i}\left( \|\widehat{\utwo}^i\|_{L_t^\infty L_x^2\cap L_t^2 H_x^1( {\Q_{1/2}(0)})}+\|\widehat{\dd}^i\|_{L_t^\infty H_x^1\cap L_t^2 H_x^2( {\Q_{1/2}(0)})} \right)<\infty.$$
	We choose a cut-off function $\varphi\in C_0^\infty(\Q_1(0))$ 
such that 
	$$0\le \varphi\le 1, \quad \varphi\equiv 1\text{ on }\Q_{\frac{1}{2}}(0), \quad \text{ and }|\pa_t\varphi|+|\nabla \varphi|+|\nabla^2 \varphi|\le C.$$
	Define $\varphi_i(x, t)=\varphi\left( \frac{x-x_i}{r_i}, \frac{t-t_i}{r_i^2} \right)$, $\forall (x, t)\in \T^3\times(0, \infty).$ Replacing $\varphi$ by $\varphi_i^2$ in the local energy inequality \eqref{eqn:LocEngy}  we can show  that
	\begin{align*}
		&\sup_{t_i-\frac{r_i^2}{4}\le t\le t_i}\int_{B_{r_i}(x_i)}\left(| {\utwo} |^2+|\nabla  {\dd} |^2+F( {\dd} )\right)\varphi_i^2 {\rd x}\\
&+\int_{\Q_{r_i/2}(z_i)}\left( |\nabla  {\utwo} |^2+|\Delta  {\dd} |^2 +|f( {\dd} )|^2 \right)\varphi_i^2 \\
	\le	& C\int_{\Q_{r_i}(z_i)}\left( | {\utwo} |^2+|\nabla  {\dd} |^2+F( {\dd} ) \right)|\pa_t \left (\varphi_i^2\right)|\\
		&+C\int_{\Q_{r_i}(z_i)}|( {\uone} + {\utwo} )|| {\uone} ||\nabla  {\utwo} ||\nabla \left (\varphi_i^2\right)|+
  | {\ptwo} || {\utwo} ||\nabla \left (\varphi_i^2 \right)|+\left( | {\utwo} |^2+|\nabla  {\dd}  |^2\right)\Delta \left (\varphi_i^2\right)
		\\
		&+C\int_{\Q_{r_i}(z_i)}|\nabla  {\dd} |^2 |\nabla^2 \left (\varphi_i^2\right)|+|  {\uone} ||\nabla  {\dd} ||\Delta  {\dd} -f( {\dd}  )|^2 |\varphi_i^2|+| {\uone} + {\utwo} ||\nabla  {\dd} |^2 |\nabla \left (\varphi_i^2\right)|\\
		&+C\int_{\Q_{r_i}(z_i)}|f( {\dd} )||\nabla  {\dd} ||\nabla \left (\varphi_i^2\right)|+|\nabla f( {\dd} )||\nabla  {\dd} | \varphi_i^2. 
	\end{align*}
 	By rescaling and Young's inequality, we can show that for every $i$, 
	\begin{align}\label{eq:young1}
		 \sup_{-\frac{1}{4}\le t\le 0}\int_{B_{\frac{1}{2}(0)}}(|\widehat{\utwo}^i|^2+|\nabla\widehat{\dd}^i|^2)+\int_{\Q_{\frac{1}{2}}(0)}(|\nabla \widehat{\utwo}^i|+|\nabla^2\widehat{\dd}^i|^2) 
		 \le C.
	\end{align}
	Hence the claim holds.  {Now we intend to apply  the Aubin–Lions lemma; up to now what we still need is to   obtain estimates on $\partial_t \widehat{\utwo}^i  $ and $\partial_t   \widehat{\dd}^i$}. To achieve this goal, 
%
{
firstly,
}
we observe that from the equations for $\widehat{\utwo}^i$ in \eqref{eqn:blowingupEQ}$_{1,2}$, we
can show that,  for
any test functions $\zeta\in C_{c, \rm{div}}^\infty (B_{\frac{1}{2}(0)};\mathbb{R}^3)$ we have 
\begin{equation}\label{eq:part1}
\begin{aligned}
 \left|\int_{Q_{\frac{1}{2}}(0)}\partial_t \widehat{\utwo}^i\cdot \zeta \right|&\le\varepsilon_i\left|\int_{Q_{\frac{1}{2}}(0)}( \widehat{\uone}^i+\widehat{\utwo}^i)\cdot((\widehat{\uone}^i+\widehat{\utwo}^i)\cdot \nabla \zeta)\right|+\left|\int_{Q_{\frac{1}{2}}(0)}\nabla \widehat{\utwo}^i:\nabla \zeta\right|\\
 &+\left|\int_{Q_{\frac12(0)}}\varepsilon_i\left(\nabla \widehat{\dd}^i\odot \nabla \widehat{\dd}^i-\frac{\varepsilon_i}{2}|\nabla \widehat{\dd}^i{\rm I_3}-\frac{r_i^2}{\varepsilon_i}F({\dd}^i){\rm I}_3\right):\nabla \zeta\right|\\
 &\le C\|\widehat{\uone}^i+\widehat{\utwo}^i\|_{L^\infty_tL_x^2(Q_{\frac12}(0))}^\frac12\|\   \widehat{\uone}^i +\widehat{\utwo}^i\|^\frac32_{L_t^2 L_x^6(Q_{\frac12}(0))}\|\nabla\zeta\|_{L_t^4{L_x^2}(Q_{\frac{1}{2}}(0))}\\
 &+C\|\widehat{\utwo}^i\|_{L_t^2H_x^1(Q_{\frac{1}{2}}(0))}\|\nabla \zeta\|_{L_t^2L_x^2(Q_{\frac12}(0))}\\
 &+C\|\nabla \widehat{\dd}^i\|_{L_t^\infty L_x^2(Q_{\frac12}(0))}^\frac12\|\nabla \widehat{\dd}^i\|_{L_t^2 L_x^6(Q_{\frac12}(0))}^\frac32\|\nabla \zeta\|_{L_t^4{L_x^2}(Q_{\frac12}(0))}\\
 &+C\|F( {\dd}^i)\|_{L_t^2 L_x^2 (Q_{\frac12}(0))}\|\nabla \zeta\|_{L_t^2 L_x^2 (Q_{\frac12}(0))}\\
 &\le C\|\zeta\|_{L_t^4{\bf V}(Q_{\frac12}(0))}, 
\end{aligned}
\end{equation}
where we have used the uniform estimates already obtained in \eqref{eq:young1}, together with the uniform $L^\infty$ bounds for ${\dd} $ and $ {\uone}$ provided in \eqref{eqn:blowupassum}.
Then from \eqref{eq:part1} we obtain the following uniform estimate:
\begin{equation}\label{eq:key1}
  \|\partial_t  \widehat{\utwo}^i    \|_{L_t^{\frac43}{\bf V}'(Q_{\frac12}(0))}\le C, \quad \forall i=1,2, ,\cdots.
\end{equation}
Now for the equations for $\widehat{\dd}^i$ in \eqref{eqn:blowingupEQ}$_{3}$, similarly, we can conclude that for any test function 
%
{$\phi\in C_0^\infty(B_{\frac{1}{2}}(0))$}
 we have
\begin{equation*}
    \begin{aligned}
        \left|\int_{Q_{\frac12}(0)}\partial_t \widehat{\dd}^i\cdot \phi\right|&\le \varepsilon_i\left|\int_{Q_{\frac12}(0)}((\widehat{\uone}^i+\widehat{\utwo}^i    )\otimes \widehat{\dd}^i):\nabla \phi \right|+\left|\int_{Q_{\frac12}(0)}\nabla \widehat{\dd}^i:\nabla \phi\right|\\
        &+\frac{r_i^2}{\varepsilon_i}\left|\int_{Q_{\frac12}(0)}f( {\dd}^i)\cdot \phi\right|\\
        &\le C\|\widehat{\uone}^i+ \widehat{\utwo}^i   \|_{L_t^2 L_x^2(Q_{\frac12}(0))}\|\widehat{\dd}^i\|_{L_t^\infty L_x^\infty(Q_{\frac12}(0))}\|\nabla \phi\|_{L_t^2 L_x^2 (Q_{\frac12}(0))}\\
        &+C\|\nabla \widehat{\dd}^i\|_{L_t^2L_x^2(Q_{\frac12}(0))}\|\nabla \phi\|_{L_t^2 L_x^2(Q_{\frac12}(0))}
\\
&+C\|f( {\dd}^i)\|_{L_t^2 L_x^2 (Q_{\frac12}(0))}\|\phi\|_{L_t^2 L_x^2 (Q_{\frac12}(0))}\\
        &\le C\|\phi\|_{L_t^2 H_x^1(Q_{\frac12}(0))}.
    \end{aligned}
\end{equation*}
This implies that 
\begin{equation}\|\partial_t   \widehat{\dd}^i \|_{L_t^2 H_x^{-1}(Q_{\frac12}(0))}\le C, \quad \forall i=1, 2, \cdots. 
\end{equation}

With the estimates on  
 $\partial_t \widehat{\utwo}^i  $ and $\partial_t   \widehat{\dd}^i$
 in hand, now we are ready to apply  the Aubin–Lions lemma. We find  that 
	$$(\widehat{\utwo}^i, \widehat{\dd}^i)\to (\widehat{\utwo}, \widehat{\dd}) \text{ strongly in }L^3(\Q_{\frac{1}{2}}(0))\times L_t^3W_x^{1, 3}(\Q_{\frac{1}{2}}(0)).$$
	We observe that $(\widehat{\utwo}, \widehat{\dd},\widehat{\ptwo})$ solves the
{linear}
system
	\begin{equation}
		\left\{
		\begin{array}{l}
			\pa_t \widehat{\utwo}+\nabla\widehat{\ptwo}=\Delta\widehat{\utwo}, \\
			\nabla\cdot \widehat{\utwo}=0, \\
			\pa_t \widehat{\dd}=\Delta \widehat{\dd}.
		\end{array}
		\right.
		\label{eqn:linearEQ}
	\end{equation}
	By the standard interior regularity estimate we
can conclude that
	$({ {\widehat{\utwo}}}, \hd)\in C^\infty(\Q_{\frac{1}{4}}(0))$ and  $\widehat{\ptwo}\in L^\infty([-\frac{1}{16}, 0], C^\infty(B_{\frac{1}{4}}(0)))$ with the following estimate:
	\begin{equation}
		\begin{aligned}
			&\tau^{-2}\int_{\Q_{\tau}(0)}(|{ {\widehat{\utwo}}}|^3+|\nabla \hd|^3)+\left( \tau^{-2}\int_{\Q_\tau(0)}|\hP|^{\frac{3}{2}} \right)^2\\
			&\le C\tau^3 \left[ \int_{\Q_{\frac{1}{2}}(0)}(|{ {\widehat{\utwo}}}|^3+|\nabla \hd|^3)+\left( \int_{\Q_{\frac{1}{2}}(0)}|\hP|^{\frac{3}{2}} \right)^2 \right]\\
			&\le C\tau^{3}, \forall \tau\in\left( 0,\frac{1}{8} \right).
		\end{aligned}
		\label{eqn:regularityLinearEq}
	\end{equation}
	From the strong convergence of $(\widehat{\utwo} ^i, \widehat{\dd}^i)$ we infer that 
	$$\tau^{-2}\int_{\Q_\tau(0)}(|{ {\widehat{\utwo}}}^i|^3+|\nabla\widehat{\dd}^i|^3)=\tau^{-2}\int_{\Q_\tau(0)}(|\widehat{\utwo}|^3+|\nabla \widehat{\dd}|^3)+\tau^{-2}o(i),$$
	where $o(i)\to 0$ as $i\to \infty$. It remains to estimate the pressure term $\hP^i$. Taking the divergence in \eqref{eqn:blowingupEQ}$_1$, we see that 
	\begin{equation}
\begin{aligned}
		&-\Delta\hP^i\\
&=\varepsilon_i\div^2\left[ {\left (\widehat{\uone}^i+\widehat{\utwo}^i\right)}\otimes {\left (\widehat{\uone}^i+\widehat{\utwo}^i\right)}+\nabla\widehat{\dd}^i\odot \nabla\widehat{\dd}^i-\frac{1}{2}|\nabla\widehat{\dd}^i|^2{\rm I}_3-\frac{r_i^2}{\ve_i^2}F({\dd}^i){\rm I}_3 \right].
		\label{eqn:blowP}
\end{aligned}
	\end{equation}
	We need to show that 
	\begin{equation}
		\tau^{-2}\int_{\Q_\tau(0)}|\hP^i|^{\frac{3}{2}}\le C\tau^{-2}(\ve_i+o(i))+C\tau.
		\label{eqn:Pesti}
	\end{equation}
	This will lead  to a contradiction of \eqref{eqn:blowingupEnergy}$_2$ if we chose sufficient small $\tau_0\in(0, \frac{1}{4})$ and sufficient large $i_0$ such that for $i\ge i_0$, it holds that 
 \begin{equation}
	\tau_0^{-2}\int_{\Q_{\tau_0}(0)}\left( |\widehat{\utwo}^i|^3+|\nabla\widehat{\dd}^i|^3 \right)+\left( \tau_0^{-2}\int_{\Q_{\tau_0}(0)}|\hP_i|^{\frac{3}{2}} \right)^2\le \frac{1}{4}. \label{eqn:Contra}
 \end{equation}
    To prove \eqref{eqn:Pesti}, let $\eta\in C_0^\infty(B_1(0))$
%
 be a cut-off function of $B_{\frac{3}{8}}(0)$. For any $t$, $-(\frac{3}{8})^2\le t\le 0$, we define 
  \begin{align*}
&\hP_{(1)}^i(x, t)\\
=&\int_{\T^3}\nabla_x^2 G(x-y)\eta(y)\varepsilon_i\left[  {\left (\widehat{\uone}^i+\widehat{\utwo}^i\right)}   \otimes {\left (\widehat{\uone}^i+\widehat{\utwo}^i\right)}+\nabla\widehat{\dd}^i\odot\nabla\widehat{\dd}^i-\frac{1}{2}|\nabla \widehat{\dd}^i|^2 {\rm I}_3-\frac{r_i^2}{\ve_i^2}F({\dd}^i){\rm I}_3 \right](y, t){\rd y},
\end{align*}
    where $G$ is the Green function of $-\Delta$ for $\T^3$. Then it is straightforward to see that $\hP_{(2)}^i(\cdot , t):=(\hP^i-\hP_{(1)}^i)(\cdot , t)$ satisfies
    \begin{equation}
      -\Delta \hP_{(2)}^i=0 \text{ in }B_{\frac{3}{8}}(0). 
      \label{eqn:hP2}
    \end{equation}
Applying the Calderon--Zygmund theory we can show that 
\begin{align}
  \left\|\hP_{(1)}^i\right\|_{L^{\frac{3}{2}}(\T^3)} &\le C\ve_i\left[ \left\|
 { \widehat{\uone}^i+\widehat{\utwo}^i }
\right\|_{L^3(B_1(0))}^2+\left\|\nabla \widehat{\dd}^i\right\|_{L^3(B_1(0))}^2+\frac{r_i^2}{\ve_i^2}\|F(\dd^i)\|_{L^{\frac{3}{2}}(B_1(0))} \right] \nonumber\\
&\le C(\ve_i+o(i)).
  \label{eqn:hP1}
\end{align}
From the mean value property of harmonic functions, we infer that for $0<\tau<\frac{1}{4}$, 
\begin{align}
\tau^{-2}\int_{\Q_\tau(0)}|\hP_{(2)}^i|^{\frac{3}{2}}&\le C\tau\int_{\Q_{\frac{1}{3}}(0)}|\hP_{(2)}^i|^{\frac{3}{2}} \nonumber\\
&\le C\tau\int_{\Q_{\frac{1}{3}}(0)}(|\hP^i|^{\frac{3}{2}}+|\hP_{(1)}^i|^{\frac{3}{2}}) \nonumber\\
&\le C\tau(1+\ve_i+o(i)).\label{eqn:hP2est}
\end{align}
Combining \eqref{eqn:hP1} and \eqref{eqn:hP2est} yields \eqref{eqn:Pesti}, and the contradiction \eqref{eqn:Contra} is achieved.  
\end{proof}

\subsection{An almost boundedness result}

By iterating Lemma \ref{lemma:smallenergyimprv} and utilizing the Riesz potential in Morrey spaces (see \cite{Adams1975,HuangWangNote2010}) we can show the following local $L^p$ estimate for all $p>1$: 
\begin{lemma}
  For any $M>0$, there exists $\ve_0>0$ depending on $M$, such that  for a suitable weak solution $({\utwo}, {\ptwo}, {\dd})$ of \eqref{eqn:SELu2} in ${\T^3}\times(0, \infty)$, if,  for $z_0=(x_0, t_0)\in {\T^3}\times(r_0^2,  \infty)$, and $r_0>0$,
   it holds that $|\uone|\le M, |{\dd}|\le M$ in $\Q_{r_0}(z_0)$,  and $\Theta_{z_0}(r_0)\le \ve_0^3$, then for any $1<p<\infty$, it holds
  true that
  $$\|(\utwo, \ptwo, \nabla {\dd})\|_{L^p(\Q_{\frac{r_0}{4}}(z_0))}\le C(p, \ve_0, M).$$
  \label{lemma:MorreyEst}
\end{lemma}
\begin{proof}
  From the assumption, we have that $\Phi_{z}(r_0/2)\le 8\ve_0^3$ holds for any $z\in \Q_{\frac{r_0}{2}}(z_0)$. Then we apply Lemma \ref{lemma:smallenergyimprv} repeatedly on $\Q_{\frac{r_0}{2}}(z)$, and we have that for any $k\ge 1$ 
 and $\tau_0 \in (0, 1/4)$,
  $$\Theta_{z}(\tau_0^k r_0)\le 2^{-k}\max\{\Theta_{z}\left( \frac{r_0}{2} \right), \frac{C_0 r_0^3}{1-2 \tau_0^3}\}.$$
  Hence, it holds for $0<s<\frac{r_0}{2}$, and 
  $z\in \Q_{\frac{r_0}{2}}(z_0)$  that 
  \begin{equation}
      s^{-2}\int_{\Q_s(z)}(|\utwo|^2+|\nabla {\dd}|^3+|\ptwo|^{\frac{3}{2}})\le C(1+\ve_0^3)\left( \frac{s}{r_0} \right)^{3\theta_0},\label{eqn:Morrey1}
  \end{equation}
  for $\theta_0=\frac{\ln 2}{3|\ln \tau_0|}\le (0, \frac{1}{3})$. 
  Now we apply the local energy inequality and boundedness of $(\uone, {\dd})$ on $\Q_{\frac{r_0}{2}}(z_0)$, for $z\in \Q_{\frac{r_0}{2}}(z_0)$ to show that 
  \begin{equation}
  \begin{aligned}
    &s^{-1}\int_{\Q_s(z)}(|\nabla {\utwo}|^2+|\nabla^2 {\dd}|^2)\\
    \le & C\left[ (2s)^{-3}\int_{\Q_{2s}(z)}(|{\utwo}|^2+|\nabla {\dd}|^2)+(2s)^{-2}\int_{\Q_{2s}(z)}(|{\utwo}|^3+|\nabla {\dd}|^2+|P|^{\frac{3}{2}}) \right]\\
    \le & C(1+\ve_0^3)\left( \frac{s}{r_0} \right)^{2\theta_0}.
  \end{aligned}\label{eqn:Morrey2}
  \end{equation}
 For any open set $U\subset \T^3\times \R_+$, $1\le p<\infty$, $0\le \lambda\le 5$, the Morrey space $M^{p, \lambda}$ is defined by 
 \begin{equation*}
     M^{p, \lambda}(U):=\left\{f\in L^p_{\rm loc}(U):\|f\|^p_{M^{p, \lambda}(U)}=\sup_{z\in U, r>0} r^{\lambda-5}\int_{\Q_r(z)}|f|^p\rd x\rd t<\infty\right\}.
 \end{equation*}
 Then \eqref{eqn:Morrey1} and \eqref{eqn:Morrey2}
read  \begin{equation*}
  	 \left\{
  	   \begin{aligned}
  	 	&({\utwo}, \nabla {\dd})\in M^{3, 3(1-\alpha)}\left( \Q_{\frac{r_0}{2}}(z_0) \right), \\
  	 	&\ptwo\in M^{\frac{3}{2}, 3(1-\alpha)}\left( \Q_{\frac{r_0}{2}}(z_0) \right), \\
  	 	&(\nabla {\utwo}, \nabla^2 {\dd})\in M^{2, 4-2\alpha}(\Q_{\frac{r_0}{2}}(z_0)).  
  	 \end{aligned}
  	 \right.
  \end{equation*}
  Now we view \eqref{eqn:SEL0}$_3$ as
  $$\pa_t{\dd}-\Delta {\dd}=-{\uu}\cdot \nabla {\dd}-f({\dd})\in M^{\frac{3}{2}, 3(1-\alpha)}\left( \Q_{\frac{r_0}{2}}(z_0) \right).$$
  Let $\eta$ be the cut off function of $\Q_{\frac{r_0}{2}}(z_0)$ with $|\pa_t \eta|+|\nabla^2 \eta|\le C r_0^{-2}$, let $\widecheck{\bf d}=\eta^2({\dd}-({\dd})_{z_0, r_0})$, 
  where $({\dd})_{z_0, r_0}$ is the average of ${\dd}$ over $\Q_{\frac{r_0}{2}}(z_0)$. Then we have 
  $$\pa_t \widecheck{\dd}-\Delta \widecheck{\bf d}=\mathsf{F}, \quad \mathsf{F}:=\eta^2 (-{\uu}\cdot \nabla {\dd}-f({\dd}))+(\pa_t \eta^2-\Delta \eta^2)({\dd}-({\dd})_{z_0, r_0})-\nabla \eta^2\cdot \nabla {\dd},$$
where {$\mathsf{F}$} satisfies 
$$\|\mathsf{F}\|_{M^{\frac{3}{2}, 3(1-\alpha)}(\R^4)}\le C(1+\ve_0).$$
By the Duhamel formula, we have that 
\begin{equation}\label{eq:du}
|\nabla\widecheck{\bf d}(x, t)|\le \int_{0}^{t}\int_{\R^3}|\nabla \Gamma(x-y, t-s)||\mathsf{F}(y, s)|dyds\le C \mathcal{I}_1(|\mathsf{F}|)(x,t),
\end{equation}
where $\Gamma$ is the heat kernel in $\R^3$, and $\mathcal{I}_\beta$ is the Riesz potential of order $\beta$ on $\R^4$, $\beta\in [0,4]$, defined by 
$$I_\beta(g)(x, t):=\int_{\R^4}\frac{|g(y, s)|}{\delta_p^{5-\beta}( (x ,t), (y, s))}dyds.$$
 {Applying}
the Riesz potential estimates (for a reference on the Riesz potential estimates, see e.g. \cite{Huang2010}), we conclude that 
\begin{equation}
    \|\nabla {\widecheck{\bf d}}\|_{M^{\frac{3(1-\alpha)}{1-2\alpha},3(1-\alpha)}(\R^4)}\le C\|\mathsf{F}\|_{M^{\frac{3}{2}, 3(1-\alpha)}(\R^4)}\le C(1+\ve_0).\label{eqn:dcheckest}
\end{equation}
Since ${\dd}-{\widecheck{\bf d}}$ is caloric in $\Q_{\frac{r_0}{2}}(z_0)$, 
\begin{equation}
    \|\nabla ({\dd}-{\widecheck{\bf d}})\|_{L^p(\Q_{\frac{r_0}{2}(z_0)})}\le C(p, r_0, \ve_0).\label{eqn:dminusdcheck}
\end{equation}
On the other hand, let $\widecheck{\utwo}$ 
solve the following Stokes equations
\begin{equation}
  \left\{
  \begin{array}{l}
    \pa_t\widecheck{\utwo}-\Delta\widecheck{\utwo}+\nabla \widecheck{\ptwo}\\
    =-\nabla\cdot[\eta^2({\uu}\otimes {\uu})+(\nabla {\dd}\odot \nabla {\dd}-\frac{1}{2}|\nabla {\dd}|^2 {\rm I}_3)-(F({\dd})-(F({\dd}))_{z_0, r_0}){\rm I}_3]\\
=:-\div \mathsf{X}, \\\\
    \nabla\cdot \widecheck{\utwo}=0, \\
    \widecheck{\utwo}(\cdot , 0)=0.
  \end{array}
  \right.
  \label{}
\end{equation}
Now we apply the Oseen kernel for estimations of $\widecheck{\utwo}$;  {for references about the 
{Oseen kernel},
 we refer the readers to e.g. \cite{Robinson2016NS, PG2023, PG2002} Particularly, by using the decaying property of the 
%
{Oseen kernel}
 (see Proposition 11.1 in \cite{PG2002}),  then with a computation similar  to \eqref{eq:du}, we can arrive at
\begin{equation}\label{eq:os}
    |\widecheck{\utwo}(x, t)|\le C\int_0^t\int_{\mathbb{R}^3}\frac{|{\mathsf{X}}(y,s)|}{\delta_p^4((x,t), (y,s))}dyds\le C \mathcal{I}_1(|\mathsf{X}|)(x, t).
\end{equation}} 
%
Note that 
\begin{align*}
  \|\mathsf{X}\|_{M^{\frac{3}{2}, 3(1-\alpha)}}\le C(1+\ve_0). 
\end{align*}
This yields that 
\begin{equation}
\|\widecheck{\utwo}\|_{M^{\frac{3(1-\alpha)}{1-2\alpha}, 3(1-\alpha)}(\R^4)}\le C\|\mathsf{X}\|_{M^{\frac{3}{2}, 3(1-\alpha)}(\R^4)}\le C(1+\ve_0).\label{eqn:vcheckest}
\end{equation}
Meanwhile, we have $({\utwo}-\widecheck{\utwo})$ solves the linear homogeneous Stokes equation in $\Q_{\frac{r_0}{2}}(z_0)$ and 
\begin{equation}
    \|{\utwo}-\widecheck{\utwo}\|_{L^p(\Q_{\frac{r_0}{4}}(z_0))}\le C(p, r_0, \ve_0).\label{eqn:vminusvcheck}
\end{equation}
Bootstrapping $\alpha$ in the previous estimates \eqref{eqn:dcheckest}, \eqref{eqn:dminusdcheck}, \eqref{eqn:vcheckest} and \eqref{eqn:vminusvcheck} to get $\alpha\uparrow \frac{1}{2}$, and then by the embedding theorem for the Morrey spaces, we obtain that 
$$\|({\utwo}, \nabla {\dd})\|_{L^p(\Q_{\frac{r_0}{4}}(z_0))}\le C(p, r_0, \ve_0).$$
The estimate for ${\ptwo}$ comes directly from the standard $L^p$ estimate for the Poisson equation:
$$-\Delta {\ptwo}=\div^2[{\uu}\otimes {\uu}+(\nabla {\dd}\odot \nabla {\dd}-\frac{1}{2}|\nabla {\dd}|^2{\rm I}_3)-(F({\dd})-(F({\dd}))_{z_0, r_0})].$$
\end{proof}

\subsection{The A-B-C-D Lemma and Criteria for the regular points} 

Now we are ready to verify the first criteria for the regular points:
\begin{proposition}\label{prop:firstReg}
  If $z_0$ satisfies the assumption in Lemma \ref{lemma:smallenergyimprv}, then $z_0$ is a regular point.  
\end{proposition}
\begin{proof}
  By Lemma \ref{lemma:MorreyEst}, we have that $({\uu}, {\ptwo}, \nabla {\dd})\in L^p(\Q_{r_0/4}(z_0))$ for $1<p<\infty$. Applying the regularity estimate for the generalized Stokes system (cf. \cite[Lemma 2.2]{jia2014local}) to the $\utwo$ equation in \eqref{eqn:SELu2}$_1$, we get that $\utwo\in C^\alpha(\Q_{\frac{r_0}{8}}(z_0))$ for some $\alpha \in (0, 1)$. On the other hand, the ${\dd}$ equation reads
  \begin{equation*}
  	\pa_t {\dd}-\Delta {\dd}=-({\uone}+{\utwo})\cdot \nabla {\dd}-f({\dd})\in L^p(\Q_{\frac{r_0}{4}}(z_0)).
  \end{equation*} 
  By the standard $W^{2, p}$ estimate we can get that $\|(\pa_t {\dd}, \nabla^2 {\dd})\|_{L^p(\Q_{\frac{r_0}{8}}(z_0))}<C(p, r_0, \ve_0)$. And the Sobolev embedding implies $\nabla {\dd}\in C^\alpha(\Q_{\frac{r_0}{8}}(z_0)),$ thus $z_0$ is a regular point. 
\end{proof}
In fact we can improve Proposition \ref{prop:firstReg} into the following theorem:
\begin{theorem}
	There exists $\varepsilon_1>0$ such that if $({\uu}, P, {\dd})$ is a suitable weak solution to \eqref{eqn:SEL0}, which satisfies, for $z_0\in \T^3\times(0, \infty)$, 
	\begin{equation}
		\limsup_{r\to 0}\frac{1}{r}\int_{\Q_r(z_0)}(|\nabla {\utwo}|^2+|\nabla^2 {\dd}|^2)\le \ve_1^2,
		\label{}
	\end{equation}
	then $({\uu}, \nabla {\dd})$ is bounded near $z_0$. 
	\label{thm:mainReq}
\end{theorem}
For simplicity, we assume that $z_0=0$. We generalize the dimensionless estimates in \cite{Caffarelli, lin1995partial} with the following quantities:
\begin{align*}
	A(r)&:=\sup_{-r^2\le t\le 0}r^{-1}\int_{B_r(0)\times\{t\}}(|{\utwo}|^2+|\nabla {\dd}|^2), \\
	B(r)&:=r^{-1}\int_{\Q_r(0)}\left( |\nabla {\utwo}|^2+|\nabla^2 {\dd}|^2 \right),\\
	C(r)&:=r^{-2}\int_{Q_r(0)}(|{\utwo}|^3+|\nabla {\dd}|^3),\\ 
	D(r)&:=r^{-2}\int_{\Q_r(0)}|{\ptwo}|^{\frac{3}{2}},
\end{align*}
To prove Theorem \ref{thm:mainReq},
we will need some preparations.   The first step is to note that we have the following interpolation lemma. 
\begin{lemma}	\label{lemma:interpolation}
	For $v\in H^1(\T^3)$, we have 
	\begin{equation}
		\begin{aligned}
			&\int_{B_r(0)}|v|^q(x, t){\rd x}\lesssim \left( \int_{B_r(0)}|\nabla v|^2(x, t){\rd x} \right)^{\frac{q}{2}-a}\left( \int_{B_r(0)}|v|^2(x, t) \right)^a\\
			&+r^{3(1-\frac{q}{2})}\left( \int_{B_r(0)}|v|^2(x, t){\rd x} \right)^{\frac{q}{2}}, 
		\end{aligned}
		\label{}
	\end{equation}
	for every $B_r(0)\subset\T^3$, $2\le q\le 6$, $a=\frac{3}{2}(1-\frac{q}{6}).$ 
\end{lemma}

With application of   Lemma \ref{lemma:interpolation} we can establish the following result.
\begin{lemma}	\label{lemma:CAB}
	For any ${\uu}\in L^\infty([-\rho^2, 0]; L^2(B_\rho(0)))\cap L^2([-\rho^2, 0]; H^1(B_\rho(0)))$, ${\dd}\in L^\infty([-\rho^2, 0]; H^1(B_\rho(0)))\cap L^2([-\rho^2, 0]; H^2(B_\rho(0)))$, it holds that for any $0<r\le \rho$,
	\begin{equation}
		C(r)\lesssim \left( \frac{r}{\rho} \right)^3 A^{\frac{3}{2}}(\rho)+\left( \frac{\rho}{r} \right)^3 A^{\frac{3}{4}}(\rho) B^{\frac{3}{4}}(\rho). 
		\label{eqn:CAB}
	\end{equation}
\end{lemma}
\begin{proof}
	By Poincar\'{e}'s inequality, we can show that for $0<r\le \rho$, 
	\begin{align*}
		&\int_{B_r(0)}(|{\utwo}|^2+|\nabla {\dd}|^2)\\
	\lesssim	& \int_{B_r(0)}\left( \left||{\utwo}|^2-(|{\utwo}|^2)\rho\right|+\left||\nabla {\dd}|^2-(|\nabla {\dd}|^2)_\rho\right| \right){\rd x}+\left( \frac{r}{\rho} \right)^3\int_{B_\rho(0)}(|{\utwo}|^2+|\nabla {\dd}|^2){\rd x}\\
	\lesssim	& \rho\int_{B_\rho(0)}(|{\utwo}||\nabla {\utwo}|+|\nabla {\dd}||\nabla^2 {\dd}|){\rd x}+\left( \frac{r}{\rho} \right)^3\int_{B_\rho(0)}(|{\utwo}|^2+|\nabla {\dd}|^2){\rd x}\\
	\lesssim 	&\rho^{\frac{3}{2}}\left( \rho^{-1}\int_{B_\rho(0)}|{\utwo}|^2+|\nabla {\dd}|^2 {\rd x} \right)^{\frac{1}{2}}\left( \int_{B_\rho(0)}|\nabla {\utwo}|^2+|\nabla^2 {\dd}|^2 \right)^{\frac{1}{2}}\\
&+\left( \frac{r}{\rho} \right)^3\int_{B_\rho(0)}\left( |{\utwo}|^2+|\nabla {\dd}|^2 \right){\rd x}\\
\lesssim		& \rho^{\frac{3}{2}}A^{\frac{1}{2}}(\rho)\left( \int_{B_\rho(0)}|\nabla {\utwo}|^2+|\nabla^2 {\dd}|^2{\rd x} \right)^{\frac{1}{2}}+\left( \frac{r}{\rho} \right)^3 A(\rho). 
	\end{align*}
	Applying Lemma \ref{lemma:interpolation} with $q=3$, $a=\frac{3}{4}$ we can get 
	\begin{align*}
		&\int_{B_r(0)}(|{\utwo}|^3+|\nabla {\dd}|^3){\rd x}\\
	\lesssim 	&\rho^{\frac{4}{3}}\left( \int_{B_r(0)}(|\nabla {\utwo}|^2+|\nabla^2 {\dd}|^2){\rd x} \right)^{\frac{3}{4}}\left( \rho^{-1}\int_{B_r(0)}(|{\utwo}|^2+|\nabla {\dd}|^2){\rd x} \right)^{\frac{3}{4}}\\
&+r^{-\frac{3}{2}}\left( \int_{B_r(0)}(|{\utwo}|^2+|\nabla {\dd}|^2){\rd x} \right)^{\frac{3}{2}}\\
	\lesssim	& \rho^{\frac{3}{4}}A^{\frac{3}{4}}(\rho)\left( \int_{B_r(0)}|\nabla {\utwo}|^2+|\nabla^2 {\dd}|^2  \right)^{\frac{3}{4}}+r^{-\frac{3}{2}}\left( \int_{B_r(0)}(|{\utwo}|^2+|\nabla {\dd}|^2) \right)^{\frac{3}{2}}\\
	\lesssim 	&\left( \rho^{\frac{3}{4}}+\frac{\rho^{\frac{9}{4}}}{r^{\frac{3}{2}}} \right)\left( \int_{B_r(0)}(|\nabla {\utwo}|^2+|\nabla^2 {\dd}|^2) \right)^{\frac{3}{4}}A^{\frac{3}{4}}(\rho)+\left( \frac{r}{\rho} \right)^3 A^{\frac{3}{2}}(\rho),
	\end{align*}
	where in the last we apply the previous inequality. Integrating the inequality over $[-r^2, 0]$ and by H\"{o}lder's inequality, we have 
	\begin{align*}
		C(r)&=\frac{1}{r^2}\int_{\Q_r(0)}(|{\utwo}|^3+|\nabla {\dd}|^3)\\
		&\lesssim \left( \frac{r}{\rho} \right)^{3}A^{\frac{3}{2}}(\rho)+\left( \rho^{\frac{3}{4}}+\frac{\rho^{\frac{9}{4}}}{r^{\frac{3}{2}}} \right)\int_{-r^2}^0\left( \int_{B_r(0)}(|\nabla {\utwo}|^2+|\nabla^2 {\dd}|^2) \right)^{\frac{3}{4}}A^{\frac{3}{4}}(\rho)\\
		&\lesssim \left( \frac{r}{\rho} \right)^3 A^{\frac{3}{2}}(\rho)+r^{-\frac{3}{2}}\rho^{\frac{3}{4}}\left( \rho^{\frac{3}{4}}+\frac{\rho^{\frac{9}{4}}}{r^{\frac{3}{2}}} \right)A^{\frac{3}{4}}(\rho)B^{\frac{3}{4}}(\rho)\\
		&=\left( \frac{r}{\rho} \right)^3 A^{\frac{3}{2}}(\rho)+\left[ \left( \frac{\rho}{r} \right)^{\frac{3}{2}}+\left( \frac{\rho}{r} \right)^3 \right]A^{\frac{3}{4}}(\rho)B^{\frac{3}{4}}(\rho)\\
		&\lesssim \left( \frac{r}{\rho} \right)^3 A^{\frac{3}{2}}(\rho)+\left( \frac{\rho}{r} \right)^3 A^{\frac{3}{4}}(\rho)B^{\frac{3}{4}}(\rho). 
	\end{align*}
\end{proof}
Next, we estimate the pressure function to get
\begin{lemma}	\label{lemma:DAB}
	Under the same assumption of Lemma \ref{lemma:CAB}, it holds that for any $0<r<\rho/2$, 
	\begin{equation}
		D(r)\lesssim \frac{r}{\rho}D(\rho)+\left( \frac{\rho}{r} \right)^2 \left[ A^{\frac{3}{4}}(\rho)B^{\frac{3}{4}}(\rho)+\rho^{\frac{3}{2}}B^{\frac{3}{4}}(\rho)+\rho^3\right]. 
		\label{eqn:DAB}
	\end{equation}
\end{lemma}
\begin{proof}
	Taking the divergence of \eqref{eqn:SELu2}$_1$ yields 
		\begin{align*}
			-\Delta {\ptwo}=&\div^2\Big[{\utwo}\otimes {\utwo}+\nabla {\dd}\odot \nabla {\dd}-\frac{1}{2}(\nabla {\dd}:\nabla {\dd}){\rm I}_3+{\uone}\otimes {\utwo}+{\utwo}\otimes {\uone}+{\uone}\otimes {\uone}-F({\dd}){\rm I}_3\Big]\\
		=	&\div^2\Big[({\utwo}-({\utwo})_\rho)\otimes ({\utwo}-({\utwo})_\rho)+(\nabla {\dd}-(\nabla {\dd})_\rho)\odot (\nabla {\dd}-(\nabla {\dd})_\rho)\\
   &-\frac{1}{2}(\nabla {\dd}-(\nabla {\dd})_\rho):(\nabla {\dd}-(\nabla {\dd})_\rho){\rm I}_3\\
			&+(\nabla {\dd}-(\nabla {\dd})_\rho)\odot (\nabla {\dd})_\rho+(\nabla {\dd})_\rho\odot (\nabla {\dd}-(\nabla {\dd})_\rho)\\
&-\frac{1}{2}(\nabla {\dd}-(\nabla {\dd})_\rho):(\nabla {\dd})_\rho {\rm I}_3\\
   &-\frac{1}{2}(\nabla {\dd})_\rho\odot (\nabla {\dd}-(\nabla {\dd})_\rho){\rm I}_3\\
			&+{\uone}\otimes ({\utwo}-({\utwo})_\rho)+({\utwo}-({\utwo})_\rho)\otimes {\uone}+{\uone}\otimes {\uone}-F({\dd}){\rm I}_3\Big]\\
		:=	&\div^2\mathcal{G} \text{ in }B_\rho. 	
		\end{align*}
	Here $({\utwo})_\rho$ denotes the average of ${\utwo}$ over $B_\rho$. 
	Let $\eta\in C_0^\infty(\T^3)$
 be a cut-off function of $B_{\rho/2}$ such that $|\nabla \eta|\lesssim \rho^{-1}$. Define an auxiliary function by
	\begin{align*}
		\widecheck{\ptwo}&(x)=-\int_{\T^3}\nabla_y^2 G(x-y):\eta^2(y) \mathcal{G}(y) {\rd y},
	\end{align*}
	where $G$ is the corresponding kernel of Laplacian on $\T^3$. Then we have
	$$-\Delta \widecheck{\ptwo}=\div^2\mathcal{G} \text{ in }B_{\rho/2},$$
	and 
	$$-\Delta ({\ptwo}-\widecheck{\ptwo})=0 \text{ in }B_{\rho/2}.$$
	From the singular integral theory (Calderon--Zygmund theory) and the $L^\infty$ bound on $({\uone}, {\dd})$ we can infer that
	\begin{align*}
		\int_{\T^3}|\widecheck{\ptwo}|^{\frac{3}{2}}{\rd x}\lesssim& \int_{\T^3}\eta^3\left( |{\utwo}-({\utwo})_\rho|^3+|\nabla {\dd}-(\nabla {\dd})_\rho|^3 \right)\\
		&+|(\nabla {\dd})_\rho|^{\frac{3}{2}}\int_{\T^3}\eta^3 |\nabla {\dd}-(\nabla {\dd})_\rho|^{\frac{3}{2}}+\int_{\T^3}\eta^3 |{\utwo}-({\utwo})_\rho|^{\frac{3}{2}}+\int_{\T^3}\eta^3\\
		\lesssim &\int_{B_\rho}(|{\utwo}-({\utwo})_\rho|^3+|\nabla {\dd}-(\nabla {\dd})_\rho|^3)+|(\nabla {\dd})_\rho|^{\frac{3}{2}}\int_{B_\rho}|\nabla {\dd}-(\nabla {\dd})_\rho|^{\frac{3}{2}}\\
&+\int_{B_\rho}(|{\utwo}-({\utwo})_\rho|^{\frac{3}{2}}+1)\end{align*}
	Notice that ${\ptwo}-\widecheck{\ptwo}$ is harmonic in $B_{\rho/2}$, and we have that for $0<r\le \rho/2$, 
	\begin{align*}
		\frac{1}{r^2}\int_{B_r}|{\ptwo}-\widecheck{\ptwo}|^{\frac{3}{2}}&\lesssim \left( \frac{r}{\rho} \right)\frac{1}{\rho^2}\int_{B_{\rho/2}}|{\ptwo}-\widecheck{\ptwo}|^{\frac{3}{2}}\\
		&\lesssim \left( \frac{r}{\rho} \right)\left[ \frac{1}{\rho^2}\int_{B_{\rho/2}}|{\ptwo}|^{\frac{3}{2}} +\frac{1}{\rho^2}\int_{B_{\rho/2}}|\widecheck{\ptwo}|^{\frac{3}{2}}\right].
	\end{align*}
	Integrating the expression above over $[-r^2, 0]$, we can show that 
	\begin{align*}
		\frac{1}{r^2}\int_{\Q_r(0)}|{\ptwo}|^{\frac{3}{2}}\lesssim& \left( \frac{r}{\rho} \right)\frac{1}{\rho^2}\int_{\Q_\rho}|{\ptwo}|^{\frac{3}{2}}+\left( \frac{\rho}{r} \right)^2\frac{1}{\rho^2}\int_{\Q_\rho}|\widecheck{\ptwo}|^{\frac{3}{2}}\\
		\lesssim& \left( \frac{r}{\rho} \right)D(\rho)+\left( \frac{\rho}{r} \right)^2 \Big[\frac{1}{\rho^2}\int_{\Q_\rho}|{\utwo}-({\utwo})_\rho|^3+|\nabla {\dd}-(\nabla {\dd})_\rho|^3\\
		&+\frac{1}{\rho^2}(\sup_{-\rho^2\le t\le 0}|(\nabla {\dd}(t))_\rho|^{\frac{3}{2}})\int_{\Q_\rho}|\nabla {\dd}-(\nabla {\dd})_\rho|^{\frac{3}{2}}\\
&+\frac{1}{\rho^2}\int_{\Q_\rho}|{\utwo}-({\utwo})_\rho|^{\frac{3}{2}} +\rho^3 \Big].
		\label{}
	\end{align*}
	Now we apply the interpolation inequality
	\begin{align*}
		&\frac{1}{\rho^2}\int_{\Q_\rho}(|{\utwo}-({\utwo})_\rho|^3+|\nabla {\dd}-(\nabla {\dd})_\rho|^3)\\
		&\lesssim\left( \sup_{-\rho^2\le t\le 0} \frac{1}{\rho}\int_{B_\rho}(|{\utwo}|^2+|\nabla {\dd}|^2)  \right)^{\frac{3}{4}}\left( \frac{1}{\rho}\int_{\Q_\rho}(|\nabla {\utwo}|^2+|\nabla^2 {\dd}|^2) \right)^{\frac{3}{4}}\\
&=A^{\frac{3}{4}}(\rho)B^{\frac{3}{4}}(\rho),
	\end{align*}
	and by H\"{o}lder's inequality and Poincar\'{e}'s inequality, we can verify that 
	\begin{align*}
		&\frac{1}{\rho^2}(\sup_{-\rho^2\le t\le 0}|(\nabla {\dd})_\rho|^{\frac{3}{2}})\int_{\Q_\rho}|\nabla {\dd}-(\nabla {\dd})_\rho|^{\frac{3}{2}}\\
		\lesssim& \rho^2\left( \sup_{-\rho^2\le t\le 0} \int_{B_\rho}|\nabla {\dd}|^2 \right)^{\frac{3}{4}}\left( \frac{1}{\rho}\int_{Q_\rho}|\nabla^2 {\dd}|^2  \right)^{\frac{3}{4}}\\
\le& A^{\frac{3}{4}}(\rho)B^{\frac{3}{4}}(\rho),
	\end{align*}
and
	\begin{align*}
		\frac{1}{\rho^2}\int_{\Q_\rho}|{\utwo}-({\utwo})_\rho|^{\frac{3}{2}}\lesssim \frac{1}{\rho^{2}}\left(\int_{\Q_\rho}|{\utwo}-({\utwo})_\rho|^2 \right)^{\frac{3}{4}}\lesssim\rho^{\frac{3}{2}} B^{\frac{3}{4}}(\rho).
	\end{align*}
\end{proof}
Now we have the interpolation estimate \eqref{eqn:CAB} and the pressure estimate \eqref{eqn:DAB}, the proof of Theorem \ref{thm:mainReq} is based on a dichotomy argument:
\begin{proof}
	[Proof of Theorem \ref{thm:mainReq}]. For $\theta\in(0, \frac{1}{2})$ and $\rho\in(0, 1)$, let $\varphi\in C_0^\infty(\Q_{\theta \rho}(0))$ be a function such that 
	$$\varphi=1 \text{ in }\Q_{\theta\rho/2}(0),\qquad |\nabla \varphi|\lesssim (\theta\rho)^{-1},\qquad |\nabla^2 \varphi|+|\pa_t \varphi|\lesssim (\theta\rho)^{-2}. $$
	Applying the local energy inequality in
\eqref{eq:ine2}
%
%
with $\varphi^2$,
the $L^\infty$-bounded of $({\uone}, {\dd})$, and integrating by parts, we obtain that 
	\begin{align*}
		&\sup_{-(\theta\rho)^2\le t\le 0}\int_{\T^3}(|{\utwo}|^2+|\nabla {\dd}|^2)\varphi^2 +\int_{\T^3\times [-(\theta \rho)^2, 0]}(|\nabla {\utwo}|^2+|\nabla^2 {\dd}|^2)\varphi^2 \\
	\lesssim	& \int_{\T^3\times[-(\theta \rho)^2, 0]}(|{\utwo}|^2+|\nabla {\dd}|^2)(|\pa_t \varphi|+|\nabla \varphi|^2+|\nabla^2 \varphi|)\\
		&+\int_{\T^3\times[-(\theta\rho)^2, 0]}[(|{\utwo}|^2-(|{\utwo}|^2)_\rho)+|{\ptwo}||{\utwo}|]|\nabla \varphi|\\
		&+\int_{\T^3\times[-(\theta \rho)^2, 0]}(|{\utwo}|^2+|\nabla {\dd}|^2) \varphi^2+ |\nabla \dd|^2|{\utwo}||\nabla \varphi|\\
&+\int_{\T^3\times[-(\theta \rho)^2, 0]}|{\uone}|^2 |\nabla \varphi|\\
		&+\int_{\T^3\times[-(\theta \rho)^2, 0]} \varphi^2. 
	\end{align*}

 Hence we obtain that 
	\begin{align*}
		&A\left( \frac{1}{2}\theta \rho \right)+B\left( \frac{1}{2}\theta \rho \right)\\
	=	&\sup_{-\left(\frac{\theta\rho}{2} \right)^2\le t\le 0}\frac{2}{\theta\rho}\int_{B_{\frac{\theta\rho}{2}}}(|{\utwo}|^2+|\nabla {\dd}|^2)+\frac{2}{\theta\rho}\int_{\Q_{\frac{\theta\rho}{2}}(0)}(|\nabla {\utwo}|^2+|\nabla^2 {\dd}|^2) \\
	\lesssim	& \sup_{-(\theta\rho)^2\le t\le 0}\frac{1}{\theta\rho}\int_{\T^3}(|{\utwo}|^2+|\nabla {\dd}|^2)\varphi^2+\frac{1}{\theta \rho}\int_{\T^3\times[-(\theta\rho)^2, 0]}(|\nabla {\uu}|^2+|\nabla^2 {\dd}|^2)\varphi^2 
	\end{align*}
	Putting together all the estimates, we have
	\begin{align*}
		&A(\frac{1}{2}\theta \rho)+B(\frac{1}{2}\theta\rho)\\
\lesssim& \left[ C^{\frac{2}{3}}(\theta\rho)+A^{\frac{1}{2}}(\theta \rho)B^{\frac{1}{2}}(\theta \rho) C^{\frac{1}{3}}(\theta \rho)+C^{\frac{1}{3}}(\theta \rho)D^{\frac{2}{3}}(\theta\rho)+(\theta \rho)^3 \right]\\
	\lesssim	& C^{\frac{2}{3}}(\theta \rho)+A(\theta\rho)B(\theta \rho)+D^{\frac{4}{3}}(\theta \rho)+(\theta\rho)^3. 
	\end{align*}
	So that 
	\begin{align*}
		A^{\frac{3}{2}}(\frac{1}{2}\theta\rho)&\lesssim C(\theta \rho)+A^{\frac{3}{2}}(\theta \rho)B^{\frac{3}{2}}(\theta \rho)+D^2(\theta\rho)+ (\theta \rho)^{\frac{9}{2}}.
	\end{align*}
	On the other hand, from Lemma \ref{lemma:CAB} and \ref{lemma:DAB} we can conclude that 
	\begin{align*}
		C(\theta\rho)&\lesssim \theta^3A^{\frac{3}{2}}(\rho)+\theta^{-3}A^{\frac{3}{4}}(\rho)B^{\frac{3}{4}}(\rho),\\
		D^2(\theta \rho)&\lesssim \theta^2 D^2(\rho)+\theta^{-4}(A^{\frac{3}{2}}(\rho)B^{\frac{3}{2}}(\rho)+\rho^3 B^{\frac{3}{2}}(\rho)+\rho^6),
	\end{align*}
	and it is easy to see that 
	$$A^{\frac{3}{2}}(\theta\rho)B^{\frac{3}{2}}(\theta \rho)\lesssim \theta^{-3}A^{\frac{3}{2}}(\rho)B^{\frac{3}{2}}(\rho).$$
	Therefore we conclude that for $0<\theta<\frac{1}{2}$, 
	\begin{align*}
		&A^{\frac{3}{2}}(\frac{1}{2}\theta\rho)+D^2(\frac{1}{2}\theta\rho)\\
		&\le c\left[ \theta^2 D^2(\rho)+\theta^{-4}(A^{\frac{3}{2}}(\rho)B^{\frac{3}{2}}(\rho)+\rho^3 B^{\frac{3}{2}}(\rho)+\rho^6)+\theta^3A^{\frac{3}{2}}(\rho)+\theta^{-8} A^{\frac{3}{2}}(\rho) B^{\frac{3}{2}}(\rho)+\theta^2+(\theta\rho)^{\frac{9}{2}} \right]\\
		&\le c(\theta^2+\theta^{-8}B^{\frac{3}{2}}(\rho))(A^{\frac{3}{2}}(\rho)+D^2(\rho))+c(\theta^{-4}\rho^3 B^{\frac{3}{2}}(\rho)+\theta^{-4}\rho^6+\theta^2+(\theta\rho)^{\frac{9}{2}}).
	\end{align*}
	For $\ve_0$ given by Lemma \ref{lemma:smallenergyimprv}, let $\theta_0\in(0, \frac{1}{2})$ be such that 
	$$c\theta_0=\min\{\frac{1}{4}, \frac{1}{8}\ve_1^2\}.$$
	Since $$\limsup_{r\to0}\frac{1}{r}\int_{\Q_r}(|\nabla {\utwo}|^2+|\nabla^2 {\dd}|^2)<\ve_1^2$$
	we can choose $\rho_0>0$ such that 
	$$c\theta_0^{-8}B^{\frac{3}{2}}(\rho)\le \frac{1}{4}\qquad \forall 0<\rho<\rho_0,$$ 
	and 
	$$c(\theta_0^{-4}\rho^3 B^{\frac{3}{2}}(\rho)+\theta_0^{-4}\rho^6+\theta_0^2+(\theta_0\rho)^{\frac{9}{2}})\le \frac{1}{2}\ve_1^2 \qquad \forall 0<\rho<\rho_0.$$
	Therefore we obtained that 
	$$A^{\frac{3}{2}}(\frac{1}{2}\theta_0 \rho)+D^2(\frac{1}{2}\theta_0 \rho)\le \frac{1}{2}(A^{\frac{3}{2}}(\rho)+D^2(\rho))+\frac{1}{2}\ve_1^2, \qquad \forall 0<\rho<\rho_0.$$
	Iterating this inequality yields that 
	$$A^{\frac{3}{2}}( (\frac{1}{2}\theta_0)^k \rho)+D^2( (\frac{1}{2}\theta_0)^k \rho)\le \frac{1}{2^k}(A^{\frac{3}{2}}(\rho)+D^2(\rho))+\ve_1^2$$
	holds for all $0<\rho<\rho_0$ and $k\ge 1$. Furthermore, we can obtain from \eqref{eqn:CAB} that 
	\begin{align*}
		C( (\frac{1}{2}\theta_0)^k \rho)&\le c[(\frac{1}{2}\theta_0)^3 A^{\frac{3}{2}}( (\frac{1}{2}\theta_0)^{k-1}\rho)+(\frac{1}{2}\theta_0)^{-3}A^{\frac{3}{4}}( (\frac{1}{2}\theta_0)^{k-1}\rho)B^{\frac{3}{4}}( (\frac{1}{2}\theta_0)^{k-1}\rho)]\\
		&\le c[(\frac{1}{2}\theta_0)^3+(\frac{1}{2}\theta_0)^{-3}\ve_1^{\frac{3}{2}}][(\frac{1}{2})^{k-1}(A^{\frac{3}{2}}(\rho)+D^2(\rho))+\ve_1^2]
	\end{align*}
	holds for all $0<\rho<\rho_0$ and $k\ge 1$. 
	Then we can arrive at 
	$$\limsup_{k\to \infty}\left[ C( (\frac{1}{2}\theta_0)^k \rho)+D^2( (\frac{1}{2}\theta_0)^k \rho) \right]\le c[1+(\frac{1}{2}\theta_0)^3+(\frac{1}{2}\theta_0)^{-3}\ve_1^{\frac{3}{2}}]\ve_1^{2}\le \frac{1}{2}\ve_0^3$$
	holds for all $\rho\in (0, \rho_0)$ provided 
that
$\ve_1=\ve_1(\theta_0, \ve_0)>0$ is chosen sufficiently small. Hence there exists a $k_0>0$ such that 
	$$C( (\frac{1}{2}\theta_0)^{k_0}\rho)+D( (\frac{1}{2}\theta_0)^{k_0}\rho)\le \ve_0^3$$
	which satisfies the assumption of Lemma \ref{lemma:smallenergyimprv}, and as a consequence, $z_0=0$ is a regular point. 
\end{proof}
\subsection{The Hausdorff measure estimate}

Now we have all the ingredients to prove Theorem \ref{thm:main2}. It is based on a standard covering argument which we detailed below for the reader's convenience. 
\begin{proof}[Proof of Theorem \ref{thm:main2}] If $z\in \Sigma$, then by Theorem \ref{thm:mainReq}, 
	\begin{equation*}
		\limsup_{r\to0}\frac{1}{r}\int_{\Q_r(z)}(|\nabla {\utwo}|^2+|\nabla^2 {\dd}|^2)\ge \ve_1.
	\end{equation*}
	Let $V$ be a neighborhood of $\Sigma$ and $\eta>0$ such that for all $z\in \Sigma$, then there exists a $r<\eta$ such that $\Q_\eta(z)\subset V$ and 
	$$ \frac{1}{r}\int_{\Q_r(z)}(|\nabla {\utwo}|^2+|\nabla^2 {\dd}|^2)\ge \ve_1. $$
	By Vitali's covering lemma, we can find a pairwisely disjoint collection of $\Q_{r_i}(z_i)$ such that 
	\begin{equation*}
		\Sigma\subset \bigcup_{i=1}^\infty \Q_{5r_i}(z_i). 
	\end{equation*}
	Hence, 
	\begin{equation}
		\begin{aligned}
			\mathcal{P}_{5\eta}^1(\Sigma)&\le\sum_{i=1}^\infty 5r_i\le \frac{5}{\ve_1}\sum_{i=1}^\infty\int_{\Q_{r_i}(z_i)} (|\nabla {\utwo}|^2+|\nabla^2 {\dd}|^2)\\
			&\le \frac{5}{\ve_1}\int_{\bigcup_i \Q_{r_i}(z_i)}(|\nabla {\utwo}|^2+|\nabla^2 {\dd}|^2)\\
			&\le \frac{5}{\ve_1}\int_V(|\nabla {\utwo}|^2+|\nabla^2 {\dd}|^2)<\infty.
		\end{aligned}\label{eqn:Hausdofest}
	\end{equation}
	
Then we can conclude that $\Sigma$ is of zero Lebesgue measure $\rd x\rd t$. Furthermore,  we can choose $V$ with arbitrarily small measure, then by the absolute continuity of integral for $|\nabla {\utwo}|^2+|\nabla^2 {\dd}|^2$, we can conclude from \eqref{eqn:Hausdofest} that  $\mathcal{P}^1(\Sigma)=0$. 
\end{proof}

\section*{Declarations}


\begin{itemize}
\item The authors declare that no funds, grants, or other support were received during the preparation of this manuscript. 
\item The authors have no relevant financial or non-financial interests to disclose.
\item The authors have no competing interests to declare that are relevant to the content of this article.
\item All authors certify that they have no affiliations with or involvement in any organization or entity with any financial interest or non-financial interest in the subject matter or materials discussed in this manuscript.
\end{itemize}
%
%
\bigskip
%
%
%
%
%
\begin{appendices}

\section{Proofs of some preliminary results}\label{secA1}


For the reader’s convenience, we include in this appendix additional details for several results whose full derivations were too lengthy to present in the main text.

We begin by providing a proof of Lemma \ref{lemma:boundz} from Section \ref{sec:def}. To this end, we first establish the following preliminary result, which is essential to the proof.

\begin{lemma}[A preliminary result]
\label{lemma:boundz1}
%
  Assuming that
\begin{equation}\label{eq:zregup1}
W(\omega)\in C^{\frac{1}{2}-\varepsilon}([0, \infty); D(A^{\frac{1}{4}+\beta})) \mbox{
$\mathbb P$-a.s. $\omega$, 
}
\end{equation} for some $\beta>\varepsilon>0$, it follows that    \eqref{eqn:zregu} holds. 
%
\end{lemma}
 \begin{proof}[Proof of Lemma  \ref{lemma:boundz1}]
From \eqref{eqn:SELu1}, we derive in the mild form  
\begin{equation*}
	{\uone}(t)=\int_0^te^{-(t-s)A}\rd W(s)=\sum_i^\infty \int_0^t e^{-(t-s)A}b_i \rd W_i(s).
\end{equation*}
By the Doob inequality and the Burkholder--Davis--Gundy inequality, for  arbitrary p, $\tilde{\delta}$, and $\varepsilon$, such that  $p>1$,  $\tilde \delta>0$, and $\varepsilon>0$, 
we have the following estimation
\begin{equation}\label{eqn:zestima}
\begin{aligned}
 & \E\left[ \sup_{0\le t\le T}\|{\uone}(t)\|_{D(A^{\tilde{\delta}+\frac{1}{2}-\varepsilon})}^{p} \right]\\
=	&\E\left[\sup_{0\le t\le T} \|A^{\tilde{\delta}+\frac{1}{2}-\varepsilon}{\uone}(t)\|_{L^2(\T^3)}^p\right]\\
 \le & C(p) \E\left( \sum_{i=1}^{\infty} \gamma_i\int_{0}^{T}\|A^{\frac{1}{2}-\varepsilon}e^{-(T-s)A} A^{\tilde{\delta}} b_i\|_{L^2(\T^3)}^2 \rd s  \right)^{\frac{p}{2}}\\
 \le & C(p)\E \left( \sum_{i=1}^{\infty}\gamma_i \int_{0}^{T}\|A^{\frac{1}{2}-\ve}e^{-(T-s)A}\|^2_{\mathcal{B}(L^2(\T^3))}\|A^{\tilde{\delta}} b_i\|_{L^2(\T^3)}^2 \rd s \right)^{\frac{p}{2}}\\
 \le & C(p)\left( \sum_{i=1}^{\infty}\gamma_i\int_0^T \frac{1}{(T-s)^{1-2\ve}}ds  \right)^{\frac{p}{2}}
\end{aligned}
\end{equation}
where $C(p)$ is a constant that only depends on $p$. 
%
%
%
This estimate immediately implies that, 
%
%
%
whenever   $W(\omega) \in  C^{1/2 - \ve} \left (  [0, T] ;  D (A^{\tilde \delta}) \right)$ for  $\mathbb P$- a.s. $\omega$, then
 it follows that 
\begin{equation}\label{eq:omega1}
 {\uone} (\omega)\in L^\infty(0, T; D(A^{{\tilde{\delta}}+\frac{1}{2}-\ve})) \mbox{
for $\mathbb P$- a.s. $\omega$}. 
\end{equation}
Now we can set $\tilde \delta= {1}/{4}+\beta$  for some $\beta>\ve>0$,  which then together with   the Sobolev embedding implies that 
\begin{equation*}
D(A^{{\tilde{\delta}}+\frac{1}{2}-\ve})
= D(A^{\frac{3}{4}+\beta-\ve})\subset H^{\frac{3}{2}+2\beta-2\ve}(\T^3)\subset L^\infty(\T^3). 
\end{equation*}
Combining this with \eqref{eq:omega1},   we conclude that \eqref{eqn:zregu} holds whenever \eqref{eq:zregup1} is satisfied.
%
%
%
\end{proof}

%
With Lemma \ref{lemma:boundz1} established, it remains to verify that the noise assumption given in \eqref{eq:assum5} indeed implies \eqref{eq:zregup1}, thereby completing the proof of Lemma \ref{lemma:boundz}.
From \eqref{eq:assum5} 
and \eqref{eqn:assum2} we have 
\begin{equation}\label{eqn:assum2p}
W(t) \in  C^{\beta} \left (  [0, T] ;  D (A^\delta) \right) \quad \mathbb P- a.s. , 
\end{equation} for some $\delta>3/4$ and  every $\beta  \in (0,  1/2)$.
Hence, 
 Hence, there exist parameters $\beta$ and $\varepsilon$ such that $\beta>\varepsilon>0$, $\beta\geq {1}/{2}-\varepsilon $ and $\delta \geq 1/4 + \beta$; 
 for example we may  choose $\varepsilon$ to be $1/4$ and $\beta$ to be in $(1/4, 1/2)$. With such a choice, \eqref{eqn:assum2p} implies \eqref{eq:zregup1}, and therefore Lemma \ref{lemma:boundz1} applies under assumption \eqref{eq:assum5}. This completes the proof of Lemma \ref{lemma:boundz}.

\section{ {Energy equality for the approximation sequence}}\label{secA2}

 {For the reader’s convenience, we present detailed calculations leading to the global energy estimates for the approximating sequence of the modified Ericksen–Leslie equations. These computations motivate the strategy for obtaining uniform estimates independent of the approximation parameter (see the proof of Lemma \ref{lem:estimates}).}

    \begin{lemma}[{Energy equality for the approximating equations}]

Fix $\omega \in \Omega$ such that  ${\uone}(\omega)\in  L_{\rm loc}^\infty(\mathbb{T}^3\times[0,\infty))$. Consider the triple $({\utwo}^\sigma(\omega), {\ptwo}^\sigma(\omega),  {\dd}^\sigma(\omega))$,  a classical strong (in the PDE sense) solution   to \eqref{eqn:appro}, then we have the following equality:
\label{lemma:LocalEnergy1} 
  \begin{equation}
    \begin{aligned}
      &\frac{{\rd}}{{\rd}t}\int_{\T^3}\left( \frac{|{\utwo}^\sigma|^2}{2}+\frac{|\nabla {\dd}^\sigma|^2}{2}+F({\dd}^\sigma) \right)\varphi+\int_{\T^3}\left( |\nabla {\utwo}^\sigma|^2+|\Delta {\dd}^\sigma|^2+|f({\dd})^\sigma|^2 \right)\varphi\\
    =  & \int_{\T^3}\left( \frac{|{\utwo}^\sigma|^2}{2}+\frac{|\nabla {\dd}^\sigma|^2}{2}+F({\dd}^\sigma) \right)\pa_t \varphi+\int_{\T^3}[({\uone}+
\Phi_\sigma[{\utwo}^\sigma])
 \cdot 
\nabla{{\utwo}}^\sigma]
\cdot {{\uone}} \varphi\\
      &+\int_{\T^3}\left(\frac{|{\utwo}^\sigma
|^2}{2}+{\utwo}^\sigma
\cdot {\uone} \right)({\uone}+
\Phi_\sigma[{\utwo}^\sigma]
)\cdot \nabla \varphi+{\ptwo} ^\sigma{\utwo}^\sigma
\cdot \nabla \varphi+\left( \frac{|{\utwo}^\sigma
|^2}{2}+\frac{|\nabla {\dd}^\sigma
|^2}{2} \right)\Delta \varphi\\
      &+\int_{\T^3}\left( \nabla {\dd}^\sigma
\odot \nabla {\dd}^\sigma
-\frac{1}{2}|\nabla {\dd}^\sigma
|^2 {\rm I}_3 \right): \nabla^2 \varphi\\
      &+\int_{\T^3}(\uone
\cdot \nabla) \Phi_\sigma[{\dd}^\sigma]
\cdot (\Delta {\dd}^\sigma-f({\dd}^\sigma))\varphi+\int_{\T^3}[({\uone}+{\utwo}^\sigma
)\cdot \nabla {\dd}^\sigma
]\cdot (\nabla \varphi\cdot \nabla ){\dd}^\sigma\\
      &-\int_{\T^3}f({\dd}^\sigma)\cdot (\nabla \varphi\cdot \nabla ){\dd}^\sigma
 {-}
2\int_{\T^3}\nabla f({\dd}^\sigma):\nabla {\dd}^\sigma\varphi. 
    \end{aligned}\label{eqn:LocEngy}
\end{equation}
 \end{lemma}
%
\begin{proof}
We first note that all the subsequent calculations are valid as  there exists 
 a classical solution to \eqref{eqn:appro} for every fixed $\sigma$, according to    \eqref{eq:smooth}. 

We take the inner product of 
\eqref{eqn:appro}$_1$ with respect to ${\utwo}^\sigma\varphi$,   and proceed to analyze each term  separately. 
We begin with the following: 
  \begin{equation*}
    \int_{\T^3}\pa_t {\utwo}^\sigma\cdot {\utwo}^\sigma \varphi=\frac{\rd}{{\rd t}}\int_{\T^3}\frac{|{\utwo^\sigma}|^2}{2}\varphi-\int_{\T^3}\frac{|{\utwo^\sigma}|^2}{2}\pa_t \varphi.
  \end{equation*}
For the second term using integration by parts we have
  \begin{align*}
    \int_{\T^3}-\Delta {\utwo}^\sigma
\cdot {\utwo}^\sigma
\varphi&=\int_{\T^3}\nabla {\utwo}^\sigma
:\nabla ({\utwo}^\sigma
\varphi)\\
    &=\int_{\T^3}|\nabla {\utwo}^\sigma
|^2 \varphi+\int_{\T^3}\nabla {\utwo}^\sigma
: ({\utwo}^\sigma
\otimes \nabla \varphi)\\
    &=\int_{\T^3}|\nabla {\utwo}^\sigma|^2 \varphi+\int_{\T^3}\nabla\left( \frac{|{\utwo}^\sigma|^2}{2} \right)\cdot \nabla \varphi\\
    &=\int_{\T^3}|\nabla {\utwo}^\sigma|^2 \varphi-\int_{\T^3}\frac{|{\utwo}^\sigma|^2}{2}\Delta \varphi.
  \end{align*}
For the third term, we have
  \begin{align*}
    &\int_{\T^3}[({\uone}+{\utwo}^\sigma
)\cdot \nabla ({\uone}+
\Phi_\sigma[{\utwo}^\sigma]
)]\cdot {\utwo}^\sigma
\varphi\\
  =  
&-
 \int_{\T^3}[  ({\uone}+{\utwo}^\sigma
)
\cdot  ({\uone}+
\Phi_\sigma[{\utwo}^\sigma]
)]\cdot {\utwo}^\sigma
\nabla \varphi 
-
\int_{\T^3}[  ({\uone}+{\utwo}^\sigma
)
\cdot  ({\uone}+
\Phi_\sigma[{\utwo}^\sigma]
)]\cdot \nabla {\utwo}^\sigma
 \varphi 
\\
 =  
&-
 \int_{\T^3}[  ({\uone}+{\utwo}^\sigma
)
\cdot  ({\uone}+
\Phi_\sigma[{\utwo}^\sigma]
)]\cdot {\utwo}^\sigma
\nabla \varphi 
-
\int_{\T^3}[   
 ({\uone}+
\Phi_\sigma[{\utwo}^\sigma]
)]\cdot \nabla {\utwo}^\sigma\cdot 
 {\uone} \varphi 
\\
&-
\int_{\T^3}[   {\utwo}^\sigma
\cdot  ({\uone}+
\Phi_\sigma[{\utwo}^\sigma]
)]\cdot \nabla {\utwo}^\sigma
 \varphi 
\\
 =  
&-
\int_{\T^3}[   
 ({\uone}+
\Phi_\sigma[{\utwo}^\sigma]
)]\cdot \nabla {\utwo}^\sigma\cdot 
 {\uone} \varphi 
-
 \int_{\T^3}[   {\uone}  
\cdot  ({\uone}+
\Phi_\sigma[{\utwo}^\sigma]
)]\cdot {\utwo}^\sigma
\nabla \varphi 
\\
&
-
 \int_{\T^3}[  {\utwo}^\sigma
\cdot  ({\uone}+
\Phi_\sigma[{\utwo}^\sigma]
)]\cdot {\utwo}^\sigma
\nabla \varphi 
-
\int_{\T^3}[  
  ({\uone}+
\Phi_\sigma[{\utwo}^\sigma]
)]\cdot \nabla \left ( \dfrac{ |{\utwo}^\sigma  |^2  }{2}\right)
 \varphi 
\\
 =  
&-
\int_{\T^3}[   
 ({\uone}+
\Phi_\sigma[{\utwo}^\sigma]
)]\cdot \nabla {\utwo}^\sigma\cdot 
 {\uone} \varphi 
-
 \int_{\T^3}[   {\uone}  
\cdot  ({\uone}+
\Phi_\sigma[{\utwo}^\sigma]
)]\cdot {\utwo}^\sigma
\nabla \varphi 
\\
&
-
\int_{\T^3}[  
  ({\uone}+
\Phi_\sigma[{\utwo}^\sigma]
)]\cdot\left ( \dfrac{ |{\utwo}^\sigma  |^2  }{2}\right)
 \nabla  \varphi 
\\
 =  
&-
\int_{\T^3}[   
 ({\uone}+
\Phi_\sigma[{\utwo}^\sigma]
)]\cdot \nabla {\utwo}^\sigma\cdot 
 {\uone} \varphi 
-
 \int_{\T^3} 
\left [   {\uone}  
\cdot    {\utwo}^\sigma
+ \left ( \dfrac{ |{\utwo}^\sigma  |^2  }{2}\right)
\right]
({\uone}+
\Phi_\sigma[{\utwo}^\sigma]
)
\nabla \varphi 
  \end{align*}
  Next, we have
  \begin{equation*}
    \int_{\T^3}\nabla {\ptwo} ^\sigma\cdot {\utwo}^\sigma \varphi=-\int_{\T^3}{\ptwo}^\sigma {\utwo} ^\sigma\cdot \nabla \varphi.
  \end{equation*}
 For the term on the right-hand side, we have:
 \begin{align}\label{eqn:canc1}
 &\int_{\T^3} [{\utwo} ^\sigma
\cdot\nabla]
\Phi_\sigma[{\dd}^
    \sigma]
\cdot (-\Delta {\dd}^\sigma+f({\dd}^\sigma))\varphi\nonumber\\
  = & \underbrace{-\int_{\T^3} [{\utwo} ^\sigma
\cdot\nabla]
\Phi_\sigma[{\dd}^
    \sigma]
\cdot \Delta {\dd}^\sigma
\varphi}_{\mathcal I} \underbrace{+
    \int_{\T^3} [{\utwo} ^\sigma
\cdot\nabla] 
\Phi_\sigma[{\dd}^
    \sigma]
\cdot f({\dd}^\sigma
)\varphi  }_{\mathcal {II}}. 
  \end{align}
  Differentiating \eqref{eqn:appro}$_3$ gives
  \begin{equation}    \label{eqn:nabladeq}
    \pa_t \nabla {\dd}^\sigma+\nabla\left[ ({\uone}+{\utwo}^\sigma
)\cdot \nabla
 \Phi_\sigma[{\dd}^
    \sigma] 
\right]=\Delta (\nabla {\dd}^\sigma
)-\nabla (f({\dd}^\sigma
)).
  \end{equation}  We take the inner product of \eqref{eqn:nabladeq} 
with   $\nabla {\dd}^\sigma\varphi$,
  and proceed to analyze each term  separately. 
We begin with the following: 
  \begin{equation*}
    \int_{\T^3}\pa_t(\nabla {\dd}^\sigma
):\nabla {\dd}^\sigma
\varphi=\frac{\rd}{{\rd t}}\int_{\T^3}\frac{1}{2}|\nabla {\dd}^\sigma
|^2 \varphi-\int_{\T^3}\frac{1}{2}|\nabla {\dd}^\sigma
|^2 \pa_t \varphi. 
  \end{equation*}
 Secondly, we have  
  \begin{align}\label{eqn:cancel2}
   & \int_{\T^3}\nabla[({\uone}+{\utwo}^\sigma
)\cdot \nabla 
\Phi_\sigma[{\dd}^
    \sigma] 
]:\nabla {\dd}
^\sigma
\varphi \nonumber\\
=&\int_{\T^3}-[({\uone}+{\utwo}^\sigma
)\cdot \nabla]
\Phi_\sigma[{\dd}^
    \sigma] 
\cdot \Delta {\dd}^\sigma
\varphi- [({\uone}+{\utwo}^\sigma
)\cdot \nabla]\Phi_\sigma[{\dd}^
    \sigma] 
\cdot [\nabla\varphi\cdot \nabla]{\dd}^\sigma\nonumber\\
  =  &- \int_{\T^3}
[({\uone}+{\utwo}^\sigma)\cdot \nabla]
\Phi_\sigma[{\dd}^
    \sigma] 
\cdot \Delta {\dd}^\sigma
\varphi
- \int_{\T^3}
[({\uone}+{\utwo}^\sigma)\cdot \nabla]\Phi_\sigma[{\dd}^
    \sigma] 
\cdot [\nabla\varphi\cdot \nabla]{\dd}^\sigma\nonumber\\
   = &\underbrace{-\int_{\T^3}[ {\uone}  \cdot \nabla]
\Phi_\sigma[{\dd}^
    \sigma] 
\cdot \Delta {\dd}
^\sigma
\varphi}_{\mathcal{III}}
   \underbrace{ -\int_{\T^3}[  {\utwo}
^\sigma
 \cdot \nabla]
\Phi_\sigma[{\dd}^
    \sigma] 
\cdot \Delta {\dd} ^\sigma
\varphi }_{\mathcal{IV}} \nonumber  \\
&   -\int_{\T^3} [({\uone}+{\utwo}
^\sigma
)\cdot \nabla]
\Phi_\sigma[{\dd}^
    \sigma] 
\cdot [\nabla\varphi\cdot \nabla]{\dd}^\sigma      .
  \end{align}
Finally, for the right-hand side term, we obtain the following after applying integration by parts:
  \begin{equation}
\label{eqn:energye1}
\begin{aligned}
   & \int_{\T^3}\nabla(\Delta {\dd}
^\sigma
-f({\dd}^\sigma)) :\nabla {\dd}
^\sigma
\varphi\\
=&-\int_{\T^3}|\Delta {\dd}
^\sigma
|^2 \varphi-\int_{\T^3} \Delta {\dd}
^\sigma
\cdot (\nabla\varphi \cdot \nabla {\dd}
^\sigma
)-\int_{\T^3}\nabla f({\dd}
^\sigma
):\nabla {\dd}
^\sigma
\varphi\\
    =&-\int_{\T^3}|\Delta {\dd}
^\sigma
|^2\varphi+\int_{\T^3}\frac{1}{2}|\nabla {\dd}
^\sigma
|^2 \Delta \varphi+\int_{\T^3}\left( \nabla {\dd}
^\sigma
\odot \nabla {\dd}
^\sigma
-\frac{1}{2}|\nabla {\dd}
^\sigma
|^2{\rm I}_3 \right):\nabla^2 \varphi\\
    &-\int_{\T^3}\nabla f({\dd}
^\sigma
):\nabla {\dd}
^\sigma
 \varphi.
     \end{aligned}
     \end{equation}
We now return to \eqref{eqn:appro}$_3$ and take the inner product of both sides with $f(\dd^\sigma)\varphi$. This yields:
\begin{align}
  \int_{\T^3}\pa_t {\dd}
^\sigma
\cdot f({\dd}^\sigma
)\varphi&=\frac{\rd}{{\rd t}}\int_{\T^3}F({\dd}^\sigma
)\varphi-\int_{\T^3}F({\dd}^\sigma
)\pa_t \varphi,\nonumber\\
  \int_{\T^3}[({\uone}+{\utwo}
^\sigma
)\cdot \nabla 
 \Phi_\sigma[{\dd}^\sigma]
%
]\cdot f({\dd}^\sigma
)\varphi &=\underbrace{\int_{\T^3}(\uone\cdot\nabla
 \Phi_\sigma[{\dd}^\sigma]
)\cdot f(\dd^\sigma
)\varphi}_{\mathcal{V}}
  \underbrace{+\int_{\T^3}(\utwo\cdot\nabla
 \Phi_\sigma[{\dd}^\sigma]
)\cdot f(\dd^\sigma
) \varphi}_{\mathcal{VI}},  \nonumber\\
  \int_{\T^3}(\Delta {\dd}^\sigma
-f({\dd}^\sigma
))\cdot f({\dd}^\sigma
)\varphi&=-\int_{\T^3}\nabla {\dd}
^\sigma:\nabla f({\dd}^\sigma
)\varphi+f({\dd}^\sigma
)\cdot (\nabla \varphi\cdot \nabla {\dd}^\sigma
)+|f({\dd}^\sigma
)|^2 \varphi.\nonumber
\end{align}
Combining
all the above identities, we obtain \eqref{eqn:LocEngy}; 
particularly, we utilize the following cancellations
\begin{align*}
\mathcal {I} - \mathcal{IV}=0,\\
\mathcal{II} - \mathcal{VI}=0.
\end{align*}
 Moreover, note that we have 
 \begin{equation*}
\mathcal{III} + \mathcal {V}=-\int_{\T^3}(\uone\cdot \nabla)
\Phi_\sigma[{\dd}^\sigma]
\cdot (\Delta {\dd}^\sigma
-f({\dd}^\sigma
))\varphi \end{equation*} 
which then yields the seventh term  on the right-hand-side of \eqref{eqn:LocEngy}.
 \end{proof}
We remark that the above computations remain valid if we choose  $\varphi$ to be a constant function and the same calculation will result in the following global energy equality:
 \begin{corollary}\label{remark:energy}
Under the
%
 {same}
assumptions as in Lemma \ref{lemma:LocalEnergy1}, it holds that
\begin{equation}
  \begin{aligned}
    &\frac{\rd}{{\rd t}}\int_{\T^3}\frac{1}{2}(|{\utwo}^\sigma|^2+|\nabla {\dd}^\sigma|^2)+F({\dd}^\sigma) {\rd x}
 +\int_{\T^3} |\nabla {\utwo}^\sigma|^2  {\rd x}
 \\
 &
+\int_{\T^3}|\Delta {\dd}^\sigma-f({\dd}^\sigma)|^2 {\rd x}
 \\
 =&\int_{\T^3}[({\uone}+\Phi_\sigma[{\utwo}^\sigma])\cdot \nabla {\utwo}^\sigma]\cdot {\uone}   {\rd x}
+\int_{\T^3} {\uone}\cdot \nabla \Phi_\sigma[{\dd}^\sigma] \cdot (\Delta {\dd}^\sigma-f({\dd}^\sigma)){\rd x}  .
  \end{aligned}
  \label{eqn:Engeqp}
\end{equation}

 \end{corollary}

\section{Details of derivations of   estimations for the approximation sequence}\label{secA3}
 

In this section, we present the detailed steps leading from \eqref{eqn:Engeq} to \eqref{eqn:GronEst}, as part of the proof of Lemma \ref{lem:estimates}. This exposition is included to facilitate the reader’s understanding of the underlying computations.

Firstly, observing that { $\|\nabla^2 {\dd}^\sigma\|_{L^2(\T^3)}^2=\|\Delta {\dd}^\sigma\|_{L^2(\T^3)}^2 $}, from  \eqref{eqn:Engeq} we have 
\begin{equation} 
  \begin{aligned}
    &\frac{\rd}{{\rd t}}\int_{\T^3}\frac{1}{2}(|{\utwo}^\sigma|^2+|\nabla {\dd}^\sigma|^2)+F({\dd}^\sigma) {\rd}x+|\nabla {\utwo}^\sigma|^2_{L^2 (\T^3)}
 \\
 &
+\dfrac{1}{8}\int_{\T^3}|\Delta {\dd}^\sigma-f({\dd}^\sigma)|^2 {\rd x}
\\
&
+
C
\int_{\T^3}[|\Delta {\dd}^\sigma|^2+|f({\dd}^\sigma)|^2 + |\nabla {\dd}^\sigma|^2|{\dd}^\sigma|^2+ |(\nabla {\dd}^\sigma)^{\top} {\dd}^\sigma|^2]{\rd x}
\\
   &\leq 
 C\left( \|{\uone}\|_{L^4(\T^3)}^4
+  {\dfrac{1}{2} 
 \| {\utwo}^\sigma\|_{L^2(\T^3)}^{2} 
}
\right)
+ C \|{\uone}\|_{L^4(\T^3)}^{4}\|{\utwo}^\sigma\|_{L^2(\T^3)}^2 
\\
&\quad + \dfrac{6}{4}\int_{\T^3} |\nabla {\dd}^\sigma|^2 {\rd x}.
  \end{aligned}
 \label{eqn:Engeq3}
\end{equation}
We define
\begin{align}
U (t) : = \int_{\T^3}\frac{1}{2}(|{\utwo}^\sigma (t)|^2+|\nabla {\dd}^\sigma (t)|^2)+F({\dd}^\sigma (t) ) {\rd}x,
\end{align}
%
Combining this with \eqref{eqn:Engeq3} and   the fact that $F({\dd}^\sigma (t))$ is nonnegative for every $t$, we obtain 
\begin{equation}  \label{eqn:Engeq24}
  \begin{aligned}
     \frac{\rd}{{\rd t}} U (t) 
 &\leq 
 C\left( \|{\uone}\|_{L^4(\T^3)}^4
+   \dfrac{1}{2}\right)
 U (t)
%
 + C \|{\uone}\|_{L^4(\T^3)}^{4} ,  
  \end{aligned}
\end{equation}
which implies that 
\begin{equation}  \label{eqn:Engeq26}
  \begin{aligned}
     \frac{\rd}{{\rd t}} U (t) 
- C\left( \|{\uone}\|_{L^4(\T^3)}^4
+   \dfrac{1}{2}\right)
 U (t)
 &\leq 
%
  C \|{\uone}\|_{L^4(\T^3)}^{4}   .
  \end{aligned}
\end{equation}
Multiplying both sides of the above inequality by 
\begin{equation*}
\exp \left({-  C \int ^t_0 \left( \|{\uone} (s)\|_{L^4(\T^3)}^4
+   \dfrac{1}{2}\right) ds}\right), 
\end{equation*} we obtain
\begin{equation*}  \label{eqn:Engeq251}
  \begin{aligned}
     \frac{\rd}{{\rd t}} 
\left ( 
e^{ {-  C \int ^t_0 \left( \|{\uone} (s)\|_{L^4(\T^3)}^4
+   \frac{1}{2}\right) ds} }
U (t) 
\right)
 &\leq 
C  e^{{-  C \int ^t_0  ( \|{\uone} (s)\|_{L^4(\T^3)}^4
+   \frac{1}{2} ) ds}}
%
  \|{\uone}\|_{L^4(\T^3)}^{4}   
 \leq C  \|{\uone}\|_{L^4(\T^3)}^{4}  . 
  \end{aligned}
\end{equation*}
Integrating both sides of
the above inequality
 from 0 to t, for any $t \in (0, T]$, 
and then   taking the supremum over $[0, T]$,
we have 
\begin{equation}  \label{eqn:Engeq251}
  \begin{aligned}
\sup_{0 \leq t \leq T}
U (t) 
%
&\leq e^{ {   C \int ^T_0 \left( \|{\uone} (s)\|_{L^4(\T^3)}^4
+   \frac{1}{2}\right) ds} } 
\left ( U(0) + 
C 
\int _0 ^T  \|{\uone} (t)\|_{L^4(\T^3)}^{4}   dt\right)
%
%
  \end{aligned}
\end{equation}
 %
Returning to \eqref{eqn:Engeq3}, we integrate both sides over the interval $[0, t]$ and subsequently take the supremum over $t \in [0, T]$, which yields
\begin{equation} 
  \begin{aligned}
%
 &\sup_{t \in [0, T]}
U (t)
+
\int_0 ^T\left (
|\nabla {\utwo}^\sigma (t)|^2_{L^2 (\T^3)}
+\dfrac{1}{8}\int_{\T^3}|\Delta {\dd}^\sigma (t)-f({\dd}^\sigma (t))|^2 {\rd x}\right) 
\rd
t
\\
&
+
C
\int_0^T
\int_{\T^3}[|\Delta {\dd}^\sigma (t)
|^2+|f({\dd}^\sigma(t)
)|^2 + |\nabla {\dd}^\sigma(t)
|^2|{\dd}^\sigma(t)
|^2
]
{\rd x}
\rd t\\
&
+ 
C\int_0^T
\int_{\T^3}
|(\nabla {\dd}^\sigma(t))^{\top} {\dd}^\sigma(t)|^2{\rd x}
\rd   t
\\
   &\leq 
 C
\int_0^T
\int_{\T^3}
\left( \|{\uone}(t)\|_{L^4(\T^3)}^4
+  {\dfrac{1}{2} 
 \| {\utwo}^\sigma(t)\|_{L^2(\T^3)}^{2} 
}
\right) {\rd x}
\rd   t
\\
&+ C\int_0^T
\left (
  \|{\uone}(t)\|_{L^4(\T^3)}^{4}\|{\utwo}(t)^\sigma\|_{L^2(\T^3)}^2  
 + \dfrac{6}{4}\int_{\T^3} |\nabla {\dd}^\sigma (t)|^2 {\rd x}
\right)\rd   t.
  \end{aligned}
 \label{eqn:Engeq355}
\end{equation}
Combining this with  \eqref{eqn:Engeq251}, we obtain
\begin{equation} 
  \begin{aligned}
%
 &
\int_0 ^T 
|\nabla {\utwo}^\sigma (t)|^2_{L^2 (\T^3)}
\rd
t
+
C
\int_0^T
\int_{\T^3}|\Delta {\dd}^\sigma (t)
|^2
{\rd x}
\rd t
\\
   &\leq 
 C
\int_0^T
\left( \|{\uone}(t)\|_{L^4(\T^3)}^4
+  {\dfrac{1}{2} 
 \| {\utwo}^\sigma(t)\|_{L^2(\T^3)}^{2} 
}
\right)
\rd   t
\\
&\quad 
+ C\int_0^T
\left (
  \|{\uone}(t)\|_{L^4(\T^3)}^{4}\|{\utwo}(t)^\sigma\|_{L^2(\T^3)}^2  
 + \dfrac{6}{4}\int_{\T^3} |\nabla {\dd}^\sigma (t)|^2 {\rd x}
\right)\rd   t
\\
 &
\leq 
 C
\int_0^T
 \|{\uone}(t)\|_{L^4(\T^3)}^4
\rd   t
+
C
T
\sup_{0 \leq t \leq T}
U (t) 
%
+ C
\sup_{0 \leq t \leq T}
U (t) 
\int_0^T
  \|{\uone}(t)\|_{L^4(\T^3)}^{4} 
\rd   t\\
&
\leq 
 C
\|{\uone}(t)\|_{L^4(Q_T)}^4
+
  e^{ {   C \int ^T_0 \left( \|{\uone} (s)\|_{L^4(\T^3)}^4
+   \frac{1}{2}\right) ds} } 
\left ( U(0) + 
C 
\|{\uone}(t)\|_{L^4(Q_T)}^4   dt\right)
%
%
,
  \end{aligned}
 \label{eqn:Engeq31}
\end{equation}
which then implies \eqref{eqn:GronEst} as desired.

\end{appendices}



\end{document}